\documentclass{article}

\usepackage{arxiv}

\usepackage[utf8]{inputenc} 
\usepackage[T1]{fontenc}    

\usepackage{hyperref}       
\usepackage{url}            
\usepackage{booktabs}       
\usepackage{amsfonts}       
\usepackage{nicefrac}       
\usepackage{microtype}      
\usepackage{graphicx}
\usepackage{natbib}
\usepackage{doi}

\usepackage{algorithm}

\usepackage{amsopn}

\usepackage{mathtools} 
\usepackage{amssymb} 
\usepackage{mathabx} 
\usepackage{siunitx}
\usepackage{mathscinet}

\usepackage{esint} 
\usepackage[bb=dsserif]{mathalpha} 

\usepackage{isomath}
\usepackage{algorithmic} 
\usepackage{textcomp} 

\usepackage{caption}
\usepackage{subcaption}

\newcommand{\R}{\mathbb{R}}

\newcommand{\Brackets}[1]{\left( #1 \right)}
\newcommand{\SquareBrackets}[1]{\left[ #1\right]}
\newcommand{\Braces}[1]{\left\{ #1\right\}}

\newcommand{\FuncAction}[2]{\left\langle #1,#2 \right\rangle}

\newcommand{\Norm}[1]{\left\lVert #1 \right\rVert}
\newcommand{\Seminorm}[1]{\left\lvert #1 \right\rvert}
\newcommand{\vertiii}[1]{{\left\vert\kern-0.25ex\left\vert\kern-0.25ex\left\vert #1\right\vert\kern-0.25ex\right\vert\kern-0.25ex\right\vert}}
\newcommand{\dx}{\,\mathrm{d}}

\providecommand{\Matrix}[1]{\mathbf{#1}}
\DeclareMathOperator{\Span}{span}
\DeclareMathOperator{\Div}{div}
\DeclareMathOperator{\Supp}{supp}
\DeclareMathOperator{\Int}{int}
\DeclareMathOperator{\Cl}{cl}
\DeclareMathOperator{\Argmin}{argmin}
\DeclareMathOperator{\Meas}{meas}

\DeclareMathSymbol{\shortminus}{\mathbin}{AMSa}{"39}

\newcommand{\GammaD}{{\Gamma_{\textnormal{D}}}}
\newcommand{\GammaN}{{\Gamma_{\textnormal{N}}}}
\newcommand{\Half}{\frac{1}{2}}

\usepackage{amsthm}
\theoremstyle{plain}
\newtheorem{theorem}{Theorem}

\newtheorem{lemma}{Lemma}
\newtheorem{corollary}{Corollary}

\newtheorem{definition}{Definition}
\newtheorem{remark}{Remark}

\usepackage{cleveref}

\title{Constraint energy minimizing generalized multiscale finite element method for inhomogeneous boundary value problems with high contrast coefficients}

\author{ \href{https://orcid.org/0000-0002-0085-6699}{\includegraphics[scale=0.06]{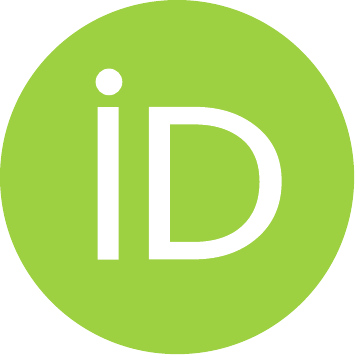}\hspace{1mm}Changqing Ye}\\
Department of Mathematics\\
The Chinese University of Hong Kong\\
Hong Kong Special Administrative Region \\
\texttt{cqye@math.cuhk.edu.hk} \\
\And
\href{https://orcid.org/0000-0002-3096-3399}{\includegraphics[scale=0.06]{orcid.pdf}\hspace{1mm}Eric T. Chung}\\
Department of Mathematics\\
The Chinese University of Hong Kong\\
Hong Kong Special Administrative Region \\
\texttt{tschung@math.cuhk.edu.hk} \\
}

\renewcommand{\shorttitle}{CEM-GMsFEM for inhomogeneous BVPs}

\hypersetup{
pdftitle={Constraint energy minimizing generalized multiscale finite element method for inhomogeneous boundary value problems with high contrast coefficients},
pdfsubject={65M12, 65M15, 65N30},
pdfauthor={Changqing YE},
pdfkeywords={Constraint energy minimization, multiscale finite element methods, high contrast problems, inhomogeneous boundary value problems},
}

\begin{document}
\maketitle

\begin{abstract}
In this article we develop the Constraint Energy Minimizing Generalized Multiscale Finite Element Method (CEM-GMsFEM) for elliptic partial differential equations with inhomogeneous Dirichlet, Neumann, and Robin boundary conditions, and the high contrast property emerges from the coefficients of elliptic operators and Robin boundary conditions. By careful construction of multiscale bases of the CEM-GMsFEM, we introduce two operators $\mathcal{D}^m$ and $\mathcal{N}^m$ which are used to handle inhomogeneous Dirichlet and Neumann boundary values and are also proved to converge independently of contrast ratios as enlarging oversampling regions. We provide a priori error estimate and show that oversampling layers are the key factor in controlling numerical errors. A series of experiments are conducted, and those results reflect the reliability of our methods even with high contrast ratios.
\end{abstract}


\section{Introduction}
Many practical problems drive us to study partial differential equations (PDE) with inhomogeneous coefficients. For example, Darcy's law in inhomogeneous or even fractured media, elasticity systems in composite materials. When coefficients show special structures, such as periodicity and stochasticity, a great number of mathematical theories have been established \cite{Bensoussan2011,DalMaso1993,Jikov1994,Pankov1997,Conca1997,Cioranescu1999,Cioranescu2008,Tartar2009,Shen2018,Armstrong2019}, which have been cornerstones of multiscale modeling and simulations. As for general inhomogeneous coefficients which are usually accompanied by high contrast channels, it has been viewed as a long-standing challenge for traditional methods. The reason is in two aspects: channelized structures require fine meshes which dramatically increase freedom degrees, and high contrast ratios deteriorate convergences of solvers for final linear systems.

To handle those problems, many multiscale computational methods have been developed since the 1990s. To name a few, multiscale finite element methods \cite{Hou1997,Hou1999,Chen2003,Efendiev2009}, Heterogeneous Multiscale Methods (HMM) \cite{Weinan2003,Weinan2005,Abdulle2012}, variational multiscale methods \cite{Hughes1995,Brezzi1997,Hughes2007}, generalized finite element methods \cite{Babuska2011,Babuska2020}, Generalized Multiscale Finite Element Methods (GMsFEM) \cite{Efendiev2013,Chung2014,Chung2016}, and Localized Orthogonal Decomposition (LOD) methods \cite{Maalqvist2014,Henning2014,Altmann2021,Hellman2017,Maalqvist2021}. A universal thought in those methods (except HMMs) is encoding fine-scale information into basis functions of Finite Element Methods (FEM), then solving original problems on multiscale finite element spaces whose dimensions have been greatly reduced compared to default FEMs. We also notice that most existing literature in multiscale computational methods chooses \emph{homogeneous} Dirichlet or Neumann Boundary Value Problems (BVP) as model problems to study convergence theories and conduct numerical experiments, while extensions of those methods to \emph{inhomogeneous} BVPs are sometimes nontrivial (e.g., \cite{Henning2014}). Since solving inhomogeneous BVPs is a practical demand from the application side, it is necessary to examine the effectiveness of existing multiscale computational methods and extend them to complex BVPs.

This work is based on Constrained Energy Minimizing Generalized Multiscale Finite Element Methods (CEM-GMsFEM), which was originally proposed in \cite{Chung2018} for high contrast problems and are applied to many applications \cite{Vasilyeva2019,Wang2021,Chung2020,Chung2018}. Note that there are two versions of CEM-GMsFEMs proposed in \cite{Chung2018}, and we focus on the modified one---\emph{relaxed} CEM-GMsFEM, which shows advantage both in theories and implementations. However, according to the construction of multiscale bases in CEM-GMsFEMs---solving energy minimizing problems on oversampling regions, either pressures or flow rates (terms from Darcy's law) of basis functions vanish on boundaries, which implies that it cannot be directly utilized in inhomogeneous BVPs. Moreover, additional physical coefficients will be introduced in Robin boundary conditions, and those coefficients may also be high contrast. Those reasons lead us to reconsider how to apply CEM-GMsFEMs to inhomogeneous BVPs in high contrast settings.

There are many common points in CEM-GMsFEMs and LOD methods, for example, both rely on exponential decay properties of bases and need mesh size-dependent oversampling regions to achieve an optimal convergence rate. We emphasize that large oversampling regions are essential here because there are theoretical evidences that reveal high contrast is strongly related to nonlocality \cite{Bellieud1998,Briane2002,Du2020} . The major difference is that CEM-GMsFEMs solve element-wise eigenvalue problems to obtain auxiliary spaces and projection operators (such an idea is originated from GMsFEMs \cite{Efendiev2013}), while LOD methods adopt quasi-interpolation operators, such as the Scott-Zhang operator \cite{Scott1990}. Since eigenvalue problems have integrated coefficient information, exponential decay rates are now explicitly dependent on $\Lambda$, where $\Lambda$ is the minimal eigenvalue that the corresponding eigenvector is not included in auxiliary spaces. From numerical experiments, $\Lambda$ is quite stable with varying contrast ratios. A minor difference from implementations is that \emph{relaxed} CEM-GMsFEMs deal with quadratic form minimization problems in constructing multiscale bases, while LOD methods need to solve saddle point problems.

The paper is organized as follows: in \cref{sec:pre}, we introduce some preliminaries; in \cref{sec:cma}, we present details of the methods, provide a rigorous analysis of numerical errors and extend the computational framework to inhomogeneous Robin BVPs; in \cref{sec:experiment}, we conduct a series of numerical experiments to verify accuracy of our methods in high contrast settings.

\section{Preliminaries}\label{sec:pre}
Denote by $\Omega \subset \R^d$ ($d=2$ or $3$) a Lipschitz domain and $\Matrix{A}(x)\in L^\infty\Brackets{\Omega;\R^{d\times d}}$ a matrix-valued function defined on $\Omega$, we consider the following model problem:
\begin{equation}\label{eq:model problem}
\left\{
\begin{aligned}
&-\Div \Brackets{\Matrix{A}\nabla u} = f \ &\text{in}\  \Omega, \\
&u = g \ &\text{on}\ \GammaD, \\
&\nu \cdot \Matrix{A}\nabla u = q \ &\text{on}\ \GammaN,
\end{aligned}
\right.
\end{equation}
where $\nu$ stands for outward unit normal vectors to $\partial \Omega$, $\GammaD$ and $\GammaN$ are two \emph{nonempty} disjointed parts of $\partial \Omega$. In this paper, we present the following assumptions for the model problem:

\paragraph{A1} There exist positive constants $0<A_1\leq A_2$ such that for a.e. $x\in \Omega$, $\Matrix{A}(x)$ is a positive define matrix with
\[
A_1\leq \lambda_{\min}\Brackets{\Matrix{A}(x)} \leq \lambda_{\max}\Brackets{\Matrix{A}(x)} \leq A_2.
\]

\paragraph{A2} The source term $f\in L^2(\Omega)$, the Dirichlet boundary value term $g\in H^{\Half}(\GammaD)$ and the Neumann boundary value term $q\in L^2(\GammaN)$.

We rewrite the inhomogeneous BVP \cref{eq:model problem} in a variational form for designing computational methods: find a solution $u_0\in V=\Braces{v\in H^1(\Omega): v=0 \ \text{on}\ \GammaD}$ such that for all $v\in V$,
\begin{equation}\label{eq:model problem varia}
\int_\Omega \Matrix{A}\nabla u_0\cdot \nabla v\dx x = \int_\Omega f v\dx x-\int_\Omega \Matrix{A} \nabla \tilde{g}\cdot \nabla v\dx x+\int_\GammaN q v\dx \sigma,
\end{equation}
where $\tilde{g} \in H^1(\Omega)$ with $\tilde{g}=g$ on $\GammaD$ in the trace sense. Obviously, $u_0$ and $u$ the solution of the original BVP have the relation $u=u_0+\tilde{g}$.

To simply notations, we denote by $a(w,v)$ the bilinear form $\int_\Omega \Matrix{A}\nabla w\cdot \nabla v \dx x$ on $V$, and $\Norm{v}_a=\sqrt{a(v,v)}$. For a subdomain $\omega \subset \Omega$, we also introduce a notation $\Norm{v}_{a(\omega)}=\sqrt{\int_\omega \Matrix{A}\nabla v\cdot \nabla v\dx x}$.

Let $\mathcal{T}^H$ be a conforming partition of $\Omega$ into elements, such as triangulations/quadrilations for 2D domains \cite{Brenner2008}, where $H$ is the coarse-mesh size to distinguish with another mesh $\mathcal{T}^h$ which will be utilized to compute multiscale basis functions, also let $N$ be the number of elements. For each $K_i\in \mathcal{T}^H$ with $1\leq i \leq N$, we define an oversampled domain $K_i^m$ ($m\geq 1$) by an iterative approach:
\[
K_i^m = \Int\Braces{\SquareBrackets{\bigcup_{\substack{K\in \mathcal{T}^H \\ \Cl(K)\cap \Cl\Brackets{K_i^{m-1}}\neq \varnothing}}\Cl\Brackets{K}} \cup \Cl\Brackets{K_i^{m-1}}},
\]
where $\Int\Brackets{S}$ and $\Cl\Brackets{S}$ are the interior and the closure of a set $S$, and we also set $K_i^0 \coloneqq K_i$ here for consistency purposes. Letting $N_{\textnormal{v}}$ be the number of vertices contained in an element (i.e., $N_\textnormal{v}=3$ for a triangular mesh and $N_\textnormal{v}=4$ for a quadrilateral mesh), we can construct a set of Lagrange bases $\Braces{\eta_i^1,\eta_i^2,\dots,\eta_i^{N_\textnormal{v}}}$ of the element $K_i\in \mathcal{T}^H$. Then we define $\tilde{\kappa}(x)$ piecewisely by
\begin{equation}\label{eq:kappa}
\tilde{\kappa}(x)\coloneqq \Brackets{N_{\textnormal{v}}-1}\sum_{j=1}^{N_\textnormal{v}} \Matrix{A}(x)\nabla \eta_i^j\cdot \nabla \eta_i^j
\end{equation}
in $K_i$. An important concept in CEM-GMsFEM is the bilinear form $s(w,v)\coloneqq \int_\Omega \tilde{\kappa} w v\dx x$, and note that $s(w,v)$ can be validly defined on $L^2(\Omega)$. Similary, we denote by $\Norm{v}_s\coloneqq \sqrt{s(v,v)}$, and $\Norm{v}_{s(\omega)}\coloneqq \sqrt{\int_\omega \tilde{\kappa} \Seminorm{v}^2 \dx x}$ for a subdomain $\omega \subset \Omega$.

The construction of the local auxiliary space $V^\textnormal{aux}_i$ is by solving an eigenvalue problem in the element $K_i$: find $\lambda_i \geq 0$ and $\phi_i\in H^1(K_i)$ such that for all $v\in H^1(K_i)$,
\[
\int_{K_i}\Matrix{A}\nabla \phi_i\cdot \nabla v\dx x=\lambda_i \int_{K_i}\tilde{\kappa} \phi_i v \dx x.
\]
We arrange the eigenvalues $\Braces{\lambda_i^j}_{j=0}^{\infty}$ in ascending order, and notice that $\lambda_i^0=0$ always holds. Denote by $V^\textnormal{aux}_i\coloneqq \Span\Braces{\phi_i^0,\phi_i^1,\dots,\phi_i^{l_i}}$, we can show that the orthogonal projection $\pi_i$ (respect to the inner product $s(\cdot,\cdot)$) from $L^2(K_i)$ onto $V^\textnormal{aux}_i$ is \footnote{We implicitly utilize a zero-extension here, which extends $V_i^\textnormal{aux}$ into $L^2(\Omega)$}
\[
\pi_i(v)\coloneqq \sum_{j=0}^{l_i} \frac{s(\phi_i^j, v)}{s(\phi_i^j, \phi_i^j)}\phi_i^j.
\]
We can immediately derive the following basic estimates: for all $v \in H^1(K_i)$,
\begin{equation} \label{eq:elem esti}
\begin{aligned}
&\Norm{v-\pi_i v}_{s(K_i)}^2 \leq \frac{\Norm{v}_{a(K_i)}^2}{\lambda_i^{l_i+1}}; \\
&\Norm{\pi_i v}_{s(K_i)}^2 = \Norm{v}_{s(K_i)}^2-\Norm{v-\pi_i v}_{s(K_i)}^2\leq \Norm{v}_{s(K_i)}^2.
\end{aligned}
\end{equation}
The global auxiliary space $V^\textnormal{aux}$ can be defined by taking a direct sum $V^\textnormal{aux}\coloneqq \oplus_i^N V^\textnormal{aux}_i$, and the global projection is $\pi\coloneqq\sum_i^N \pi_i$ accordingly.

Although \cref{eq:elem esti} shows that $V^\textnormal{aux}$ can approximate $V$ satisfyingly and stably with respect to the contrast ratio $A_2/A_1$, functions in $V^\textnormal{aux}$ may not be continuous in $\Omega$, and thus $V^\textnormal{aux}$ cannot be used as a conforming finite element space. The essential thought in CEM-GMsFEMs is ``extending'' $\phi_i^j$ into $V$ by solving an energy minimization problem:
\begin{equation}\label{eq:global basis}
\psi_i^j =\Argmin\Braces{a(\psi, \psi)+s(\pi \psi-\phi_i^j, \pi \psi-\phi_i^j):\psi \in V},
\end{equation}
which is a relaxed version of the following problem:
\[
\Argmin\Braces{a(\psi, \psi):\psi \in V, \pi \psi=\phi_i^j}.
\]
Moreover, it can be proved that $\psi_i^j$ decays exponentially fast away from $K_i$, which implies that solving \cref{eq:global basis} on an oversampling domain is reasonable. Denote by $V_i^m$ the space
\[
\Braces{v\in H^1(K_i^m): v=0 \ \text{on}\ \GammaD \cap \partial K_i^m\ \text{or}\ \Omega \cap \partial K_i^m},
\]
we can see that the zero-extension of a function in $V_i^m$ still belongs to $V$. Then the multiscale basis function $\psi_i^{j,m}$ is defined as follows:
\[
\psi_i^{j,m} =\Argmin\Braces{a(\psi, \psi)+s(\pi \psi-\phi_i^j, \pi \psi-\phi_i^j):\psi \in V_i^m}.
\]
It can be shown that $\psi_i^j$ and $\psi_i^{j,m}$ satisfy a variational form respectively:
\begin{equation}
a(\psi_i^j, v)+s(\pi \psi_i^j, \pi v)=s(\phi_i^j, \pi v)\quad \forall v \in V;
\end{equation}
\begin{equation}\label{eq:local basis varia}
a(\psi_i^{j,m}, v)+s(\pi \psi_i^{j,m}, \pi v)=s(\phi_i^j, \pi v)\quad \forall v \in V_i^m.
\end{equation}
We also introduce notations $V^\textnormal{glo}_\textnormal{ms}\coloneqq \Span\Braces{\psi_i^j}_{0\leq j\leq l_i, 1\leq i\leq N}$ and $V^m_\textnormal{ms}\coloneqq \Span\Braces{\psi_i^{j,m}}_{0\leq j\leq l_i, 1\leq i\leq N}$. A basis property of $V^\textnormal{glo}_\textnormal{ms}$ is the orthogonality to the kernel of $\pi$ with respect to the inner product $a(\cdot,\cdot)$.
\begin{lemma}[\cite{Chung2018}]\label{lem:basic orth}
Let $v\in V^\textnormal{glo}_\textnormal{ms}$, then $a(v, v')=0$ for any $v'\in V$ with $\pi v'=0$. Moreover, if there exists $v\in V$ such that $a(v,v')=0$ for any $v'\in V^\textnormal{glo}_\textnormal{ms}$, then $\pi v=0$.
\end{lemma}

\section{Computational method and analysis}\label{sec:cma}
\subsection{Method}
The computational method for solving the model problem \cref{eq:model problem} consists of four steps:

\paragraph{Step1} Find $\mathcal{D}^m_i \tilde{g}\in V_i^m$ and $\mathcal{N}^m_i q \in V_i^m$ such that for all $v\in V_i^m$,
\begin{align}
a(\mathcal{D}^m_i \tilde{g}, v)+s(\pi \mathcal{D}^m_i \tilde{g}, \pi v)&=\int_{K_i} \Matrix{A}\nabla \tilde{g} \cdot \nabla v\dx x,  \label{eq:local diri} \\
a(\mathcal{N}^m_i q, v)+s(\pi \mathcal{N}^m_i q, \pi v)&=\int_{\partial K_i \cap \GammaN} qv\dx \sigma. \label{eq:local neum}
\end{align}
Then take summations as $\mathcal{D}^m \tilde{g} = \sum_{i=1}^N \mathcal{D}^m_i \tilde{g}$ and $\mathcal{N}^m q=\sum_{i=1}^N \mathcal{N}^m_i q$.

\paragraph{Step2} Prepare the multiscale function space $V^m_\textnormal{ms}$ via \cref{eq:local basis varia}.

\paragraph{Step3} Solve $w^m \in V^m_\textnormal{ms}$ such that for all $v \in V^m_\textnormal{ms}$,
\begin{equation}\label{eq:main solver}
a(w^m, v)=\int_\Omega f v\dx x-\int_\Omega \Matrix{A}\nabla\tilde{g}\cdot \nabla v\dx x+\int_\GammaN q v\dx \sigma+a(\mathcal{D}^m \tilde{g},v)-a(\mathcal{N}^m q,v).
\end{equation}

\paragraph{Step4} Construct the numerical solution to approximate the real solution of \cref{eq:model problem} as
\[
u \approx w^m-\mathcal{D}^m \tilde{g}+\mathcal{N}^m q+\tilde{g}.
\]

Note that \cref{eq:local basis varia,eq:local diri,eq:local neum} are all solved on a fine mesh $\mathcal{T}^h$, which is a refinement of $\mathcal{T}^H$. For brevity, we will not explicitly point this out in here and following analysis parts. From the computational steps presented above, the multiscale finite element space $V^m_\textnormal{ms}$ is reusable for different source terms, which will greatly accelerate computations in simulations. Moreover, if several particular boundary values (i.e., $g$ and $q$) that we are interested in admit a low-dimension structure, it is also possible to build abstract operators $\mathcal{D}^m$ and $\mathcal{N}^m$ to achieve an overall saving in computational resources.

\subsection{Analysis} \label{subsec:analysis}
We can also define ``global'' versions of $\mathcal{D}^m$, $\mathcal{N}^m$ and $w^m$ as $\mathcal{D}^\textnormal{glo}\coloneqq \sum_{i=1}^N \mathcal{D}^\textnormal{glo}_i$, $\mathcal{N}^\textnormal{glo}\coloneqq \sum_{i=1}^N \mathcal{N}^\textnormal{glo}_i$ and $w^\textnormal{glo}$ respectively, where $\mathcal{D}^\textnormal{glo}_i$ satisfies for all $v\in V$
\begin{equation}\label{eq:glob diri}
a(\mathcal{D}^\textnormal{glo}_i \tilde{g}, v)+s(\pi \mathcal{D}^\textnormal{glo}_i \tilde{g}, \pi v)=\int_{K_i} \Matrix{A}\nabla \tilde{g} \cdot \nabla v\dx x\quad \forall,
\end{equation}
$\mathcal{N}^\textnormal{glo}_i$ satisfies for all $v\in V$
\begin{equation}\label{eq:glob neum}
a(\mathcal{N}^\textnormal{glo}_i q, v)+s(\pi \mathcal{N}^\textnormal{glo}_i q, \pi v)=\int_{\partial K_i \cap \GammaN} qv\dx \sigma\quad \forall v \in V.
\end{equation}
and $w^\textnormal{glo}$ satisfies for all $v \in V^\textnormal{glo}_\textnormal{ms}$
\begin{equation} \label{eq:glob w}
a(w^\textnormal{glo}, v)=\int_\Omega f v\dx x-\int_\Omega \Matrix{A}\nabla\tilde{g}\cdot \nabla v\dx x+\int_\GammaN q v\dx \sigma+a(\mathcal{D}^\textnormal{glo}\tilde{g},v)-a(\mathcal{N}^\textnormal{glo} q,v).
\end{equation}

The starting point of analyzing CEM-GMsFEMs is that the ``global'' solution possess an optimal error estimate:
\begin{theorem}\label{thm:glob error}
Let $\mathcal{D}^\textnormal{glo}_i \tilde{g}$, $\mathcal{N}^\textnormal{glo}_i q$, $w^\textnormal{glo}$ be the solutions of \cref{eq:glob diri,eq:glob neum,eq:glob w} respectively, $u$ be the real solution of the model problem \cref{eq:model problem}. Then
\begin{equation}
\Norm{w^\textnormal{glo}-\mathcal{D}^\textnormal{glo}\tilde{g}+\mathcal{N}^\textnormal{glo}q+\tilde{g}-u}_a\leq \frac{1}{\sqrt{\Lambda}} \Norm{f}_{s^{-1}},
\end{equation}
where
\[
\Norm{f}_{s^{-1}}\coloneqq \sup_{\substack{v\in L^2(\Omega)\\ v\neq 0}} \frac{\int_\Omega f v\dx x}{\Norm{v}_s},
\]
$\mathcal{D}^\textnormal{glo}\coloneqq \sum_{i=1}^N \mathcal{D}^\textnormal{glo}_i$, $\mathcal{N}^\textnormal{glo}\coloneqq \sum_{i=1}^N \mathcal{N}^\textnormal{glo}_i$ and $\Lambda= \min_i \lambda_i^{l_i+1}$.
\end{theorem}
\begin{proof}
For simplicity, take $e=w^\textnormal{glo}-\mathcal{D}^\textnormal{glo}\tilde{g}+\mathcal{N}^\textnormal{glo}q+\tilde{g}-u$. By a direct computation, it gives for all $v\in V^\textnormal{glo}_\textnormal{ms}$,
\[
a(e,v)=a(w^\textnormal{glo},v)-a(\mathcal{D}^\textnormal{glo}\tilde{g},v)+a(\mathcal{N}^\textnormal{glo}q,v)-a(u_0,v)\overset{\cref{eq:model problem varia}}{=}0.
\]
We hence derive that $\pi e=0$ via \cref{lem:basic orth}. Moreover, for all $v\in V$ with $\pi v=0$, we have $a(w^\textnormal{glo},e)=0$ and
\[
a(e,v)=-a(\mathcal{D}^\textnormal{glo}\tilde{g},v)+a(\mathcal{N}^\textnormal{glo}q,v)-a(u_0,v).
\]
From the definitions of $\mathcal{D}^\textnormal{glo}$ and $\mathcal{N}^\textnormal{glo}$ (see \cref{eq:glob diri,eq:glob neum}) and recalling $\pi v=0$, we can show
\[
a(\mathcal{D}^\textnormal{glo}\tilde{g},v)=\int_\Omega \Matrix{A}\nabla \tilde{g}\cdot \nabla v\dx x \quad \text{and}\quad a(\mathcal{N}^\textnormal{glo}q,v)=\int_\GammaN q v \dx \sigma.
\]
Combining the variation form of $u_0$ \cref{eq:model problem varia}, we conclude that for $v\in V$ with $\pi v=0$
\[
a(e,v) = \int_\Omega f v\dx x.
\]
It follows that
\[
\Norm{e}_a^2=\int_\Omega f e\dx x\leq \Norm{f}_{s^{-1}}\Norm{e}_s=\Norm{f}_{s^{-1}}\Norm{e-\pi e}_s\leq \frac{1}{\sqrt{\Lambda}}\Norm{f}_{s^{-1}}\Norm{e}_a.
\]
\end{proof}

The analysis of CEM-GMsFEM can be summarized in the following abstract problem:
\paragraph{Abstract problem} Let $K_i\in \mathcal{T}^H$ and $t_i \in V'$ such that $\FuncAction{t_i}{v}=0$ holds for any $v\in V$ with $\Supp(v)\subset \Omega \setminus K_i$; define an operator $\mathcal{P}_i:V'\rightarrow V$ with $\mathcal{P}_i t_i$ satisfying for all $v\in V$,
\begin{equation} \label{eq:ab glob varia}
a(\mathcal{P}_i t_i, v)+s(\pi \mathcal{P}_i t_i, \pi v)=\FuncAction{t_i}{v};
\end{equation}
similarly, define $\mathcal{P}^m_i$ and $\mathcal{P}^m_i t_i$ such that for all $v\in V_i^m$,
\begin{equation}\label{eq:ab loca varia}
a(\mathcal{P}^m_i t_i, v)+s(\pi \mathcal{P}^m_i t_i, \pi v)=\FuncAction{t_i}{v}.
\end{equation}
The goal is deriving an estimate on
\[
\Norm{\sum_{i=1}^N\Brackets{\mathcal{P}_i-\mathcal{P}^m_i}t_i}_a^2+\Norm{\pi\sum_{i=1}^N\Brackets{\mathcal{P}_i-\mathcal{P}^m_i}t_i}_s^2.
\]

We complete such an estimate by several lemmas. The first lemma shows that $\mathcal{P}_i t_i$ propose an exponentially decaying property.
\begin{lemma}\label{lem:ab lem 1}
Let $\Lambda = \min_i \lambda_i^{l_i+1}$ and $m\geq 1$, there exists a positive constant $0<\theta<1$ such that
\[
\Norm{\mathcal{P}_i t_i}_{a\Brackets{\Omega \setminus K_i^m}}^2+\Norm{\pi \mathcal{P}_i t_i}_{s\Brackets{\Omega\setminus K_i^m}}^2 \leq \theta^m \Brackets{\Norm{\mathcal{P}_i t_i}_{a}^2+\Norm{\pi \mathcal{P}_i t_i}_{s}^2},
\]
\end{lemma}
where $\theta = \frac{c_{*}}{c_{*}+1}$ and
\[
c_{*}(\Lambda)=\max_{x\in [0,\frac{\pi}{2}]} \Brackets{\cos(x)+\sin(x)}\Brackets{\frac{\cos(x)}{\sqrt{\Lambda}}+\sin(x)}.
\]

The next lemma bounds the error between $\mathcal{P}_i t_i$ and $\mathcal{P}_i^m t_i$.
\begin{lemma}\label{lem:ab lem 2}
Keep the notations same as \cref{lem:ab lem 1}, then
\[
\Norm{\Brackets{\mathcal{P}_i-\mathcal{P}_i^m}t_i}_a^2+\Norm{\pi \Brackets{\mathcal{P}_i-\mathcal{P}_i^m}t_i}_s^2 \leq c_{\star}\theta^{m-1}\Brackets{\Norm{\mathcal{P}_i t_i}_{a}^2+\Norm{\pi \mathcal{P}_i t_i}_{s}^2},
\]
where
\[
c_{\star}(\Lambda)=\max_{x\in [0,\frac{\pi}{2}]} \SquareBrackets{\Brackets{\frac{1}{\sqrt{\Lambda}}+1}\cos(x)+\sin(x)}^2+\Brackets{\frac{\cos(x)}{\sqrt{\Lambda}}+\sin(x)}^2.
\]
\end{lemma}

We need a regularity assumption for $\mathcal{T}^H$ before presenting the concluding lemma.

\paragraph{A3} There exists a positive constant $C_\textnormal{ol}$ such that for all $K_i\in \mathcal{T}^H$ and $m > 0$,
\[
\#\Braces{K\in \mathcal{T}^H: K\subset K_i^m} \leq C_\textnormal{ol} m^d.
\]

\begin{lemma}\label{lem:ab lem 3}
Keep the notations same as \cref{lem:ab lem 1,lem:ab lem 2}, then
\[
\Norm{\sum_{i=1}^N\Brackets{\mathcal{P}_i-\mathcal{P}^m_i}t_i}_a^2+\Norm{\pi\sum_{i=1}^N\Brackets{\mathcal{P}_i-\mathcal{P}^m_i}t_i}_s^2\leq c_\star^2 C_{\textnormal{ol}} \theta^{m-1}(m+1)^d \sum_{i=1}^N \FuncAction{t_i}{\mathcal{P}_i t_i}.
\]
\end{lemma}

A common trick in proving those lemmas is multiplying a cutoff function $\chi$ to $v$ and then inserting $\chi v$ into variational forms. Here is the definition of cutoff functions.
\begin{definition}\label{def:cutoff}
Let $V^H$ be the Lagrange basis function space of $\mathcal{T}^H$. For an element $K_i \in \mathcal{T}^H$, a cutoff function $\chi_i^{n,m}\in V^H$ satisfies properties:
\begin{equation}
\begin{aligned}
\chi_i^{n,m}(x) \equiv 1\quad &\text{in}\ K_i^n;\\
\chi_i^{n,m}(x) \equiv 0\quad &\text{in}\ \Omega\setminus K_i^{m};\\
0\leq \chi_i^{n,m}(x)\leq 1\quad &\text{in}\ K_i^{m}\setminus K_i^{n}.
\end{aligned}
\end{equation}
\end{definition}

We start to prove \cref{lem:ab lem 1,lem:ab lem 2,lem:ab lem 3} now.
\begin{proof}[Proof of \cref{lem:ab lem 1}]
Replacing $v$ with $\Brackets{1-\chi_i^{m-1,m}} \mathcal{P}_i t_i$ in \cref{eq:ab glob varia}, recalling $1-\chi_i^{m-1,m}\equiv 0$ in $K_i^{m-1}$ and $1-\chi_i^{m-1,m}\equiv 1$ in $\Omega\setminus K_i^{m}$, we then obtain
\[
\begin{aligned}
&\Norm{\mathcal{P}_i t_i}_{a\Brackets{\Omega \setminus K_i^m}}^2+\Norm{\pi \mathcal{P}_i t_i}_{s\Brackets{\Omega\setminus K_i^m}}^2 \\
=&\int_{K_i^m\setminus K_i^{m-1}}\Brackets{\chi_i^{m-1,m}-1}\Matrix{A}\nabla \mathcal{P}_i t_i\cdot \nabla \mathcal{P}_i t_i\dx x\\
&\quad+\int_{K_i^m\setminus K_i^{m-1}} \mathcal{P}_i t_i \Matrix{A}\nabla \mathcal{P}_i t_i \cdot \nabla \chi_i^{m-1,m}\dx x\\
&\quad+\int_{K_i^m\setminus K_i^{m-1}} \tilde{\kappa} \pi \mathcal{P}_i t_i\cdot \pi \SquareBrackets{\Brackets{\chi_i^{m-1,m}-1}\mathcal{P}_i t_i} \dx x\\
\coloneqq& I_1+I_2+I_3.
\end{aligned}
\]
According to \cref{def:cutoff}, we have $\chi_i^{m-1,m}-1\leq 0$ in $K_i^m\setminus K_i^{m-1}$, which gives $I_1\leq 0$. By the definition of $\tilde{\kappa}$ in \cref{eq:kappa}, it is easy to show
\[
\Matrix{A}(x)\nabla \chi_i^{m-1,m}\cdot \nabla \chi_i^{m-1,m} \leq \tilde{\kappa}(x)
\]
in $K_i^{m}\setminus K_i^{m-1}$. Then, we could derive
\[
I_2 \leq \Norm{\mathcal{P}_i t_i}_{a\Brackets{K_i^{m}\setminus K_i^{m-1}}}\Norm{\mathcal{P}_i t_i}_{s\Brackets{K_i^{m}\setminus K_i^{m-1}}}.
\]
For $I_3$, applying the Cauchy–Schwarz inequality and estimates \eqref{eq:elem esti}, we have
\[
\begin{aligned}
I_3 &\leq \Norm{\pi \mathcal{P}_i t_i}_{s\Brackets{K_i^{m}\setminus K_i^{m-1}}} \Norm{\pi \SquareBrackets{\Brackets{\chi_i^{m-1,m}-1}\mathcal{P}_i t_i}}_{s\Brackets{K_i^{m}\setminus K_i^{m-1}}} \\
&\leq \Norm{\pi \mathcal{P}_i t_i}_{s\Brackets{K_i^{m}\setminus K_i^{m-1}}} \Norm{\Brackets{\chi_i^{m-1,m}-1}\mathcal{P}_i t_i}_{s\Brackets{K_i^{m}\setminus K_i^{m-1}}} \\
&\leq \Norm{\pi \mathcal{P}_i t_i}_{s\Brackets{K_i^{m}\setminus K_i^{m-1}}} \Norm{\mathcal{P}_i t_i}_{s\Brackets{K_i^{m}\setminus K_i^{m-1}}}
\end{aligned}
\]
Meanwhile, \eqref{eq:elem esti} provide an estimate for $\Norm{\mathcal{P}_i t_i}_{s\Brackets{K_i^{m}\setminus K_i^{m-1}}}$ as
\[
\begin{aligned}
\Norm{\mathcal{P}_i t_i}_{s\Brackets{K_i^{m}\setminus K_i^{m-1}}}\leq & \Norm{\mathcal{P}_i t_i-\pi \mathcal{P}_i t_i}_{s\Brackets{K_i^{m}\setminus K_i^{m-1}}}+\Norm{\pi \mathcal{P}_i t_i}_{s\Brackets{K_i^{m}\setminus K_i^{m-1}}} \\
\leq & \frac{1}{\sqrt{\Lambda}} \Norm{\mathcal{P}_i t_i}_{a\Brackets{K_i^{m}\setminus K_i^{m-1}}}+\Norm{\pi \mathcal{P}_i t_i}_{s\Brackets{K_i^{m}\setminus K_i^{m-1}}}.
\end{aligned}
\]
Collecting all the estimates for $I_1$, $I_2$ and $I_3$, we arrive at
\[
\begin{aligned}
&\Norm{\mathcal{P}_i t_i}_{a\Brackets{\Omega \setminus K_i^m}}^2+\Norm{\pi \mathcal{P}_i t_i}_{s\Brackets{\Omega\setminus K_i^m}}^2 \\
\leq & \Brackets{\Norm{\mathcal{P}_i t_i}_{a\Brackets{K_i^{m}\setminus K_i^{m-1}}}+\Norm{\pi \mathcal{P}_i t_i}_{s\Brackets{K_i^{m}\setminus K_i^{m-1}}}}\Brackets{\frac{1}{\sqrt{\Lambda}} \Norm{\mathcal{P}_i t_i}_{a\Brackets{K_i^{m}\setminus K_i^{m-1}}}+\Norm{\pi \mathcal{P}_i t_i}_{s\Brackets{K_i^{m}\setminus K_i^{m-1}}}} \\
\leq & c_{*}(\Lambda)\Brackets{\Norm{\mathcal{P}_i t_i}_{a\Brackets{K_i^{m}\setminus K_i^{m-1}}}^2+\Norm{\pi \mathcal{P}_i t_i}_{s\Brackets{K_i^{m}\setminus K_i^{m-1}}}^2}.
\end{aligned}
\]
This yields an iterative relation
\[
\begin{aligned}
&\Norm{\mathcal{P}_i t_i}_{a\Brackets{\Omega \setminus K_i^{m-1}}}^2+\Norm{\pi \mathcal{P}_i t_i}_{s\Brackets{\Omega\setminus K_i^{m-1}}}^2\\
\geq&\Brackets{1+\frac{1}{c_{*}}} \Norm{\mathcal{P}_i t_i}_{a\Brackets{\Omega \setminus K_i^{m}}}^2+\Norm{\pi \mathcal{P}_i t_i}_{s\Brackets{\Omega\setminus K_i^{m}}}^2,
\end{aligned}
\]
and also finishes the proof.
\end{proof}

\begin{proof}[Proof of \cref{lem:ab lem 2}]
Let $z_i\coloneqq \Brackets{\mathcal{P}_i-\mathcal{P}_i^m}t_i$, and decompose $z_i$ as
\[
z_i=\Braces{\Brackets{1-\chi_i^{m-1,m}}\mathcal{P}_i t_i}+\Braces{\Brackets{\chi_i^{m-1,m}-1}\mathcal{P}_i^m t_i + \chi_i^{m-1,m}z_i}\coloneqq z_i'+z_i''.
\]
Recalling the definition of $\chi_i^{m-1,m}$, we have $z_i''\in V_i^{m}$. Then combining \cref{eq:ab glob varia,eq:ab loca varia}, we can obtain
\[
a(z_i,z_i'')+s(\pi z_i,\pi z_i'')=0.
\]
Techniques for estimating $a(z_i,z_i')+s(\pi z_i,\pi z_i'')$ are basically same as the proof of \cref{lem:ab lem 1}:
\[
\Norm{z_i}_a^2+\Norm{\pi z_i}_s^2=a(z_i,z_i')+s(\pi z_i,\pi z_i') \leq \Norm{z_i}_a \Norm{z_i'}_a+\Norm{\pi z_i}_s\Norm{\pi z_i'}_s;
\]
\[
\begin{aligned}
\Norm{z_i'}_a =& \Norm{\Brackets{1-\chi_i^{m-1,m}}\mathcal{P}_i t_i}_a \leq \Norm{\mathcal{P}_i t_i}_{a(\Omega\setminus K_i^{m-1})}+\Norm{\mathcal{P}_i t_i}_{s(\Omega\setminus K_i^{m-1})} \\
\leq&\Brackets{1+\frac{1}{\sqrt{\Lambda}}}\Norm{\mathcal{P}_i t_i}_{a(\Omega\setminus K_i^{m-1})}+\Norm{\pi \mathcal{P}_i t_i}_{s(\Omega\setminus K_i^{m-1})};
\end{aligned}
\]
\[
\begin{aligned}
\Norm{\pi z_i'}_s&=\Norm{\pi\SquareBrackets{\Brackets{1-\chi_i^{m-1,m}}\mathcal{P}_i t_i}}_s\leq \Norm{\Brackets{1-\chi_i^{m-1,m}}\mathcal{P}_i t_i}_s \leq \Norm{\mathcal{P}_i t_i}_{s(\Omega \setminus K_i^{m-1})}  \\
&\leq \frac{1}{\sqrt{\Lambda}}\Norm{\mathcal{P}_i t_i}_{a(\Omega\setminus K_i^{m-1})}+\Norm{\pi \mathcal{P}_i t_i}_{s(\Omega\setminus K_i^{m-1})}
\end{aligned}
\]
We then obtain
\[
\begin{aligned}
&\Norm{z_i}_a^2+\Norm{\pi z_i}_s^2=a(z_i,z_i')+s(\pi z_i,\pi z_i')\\
\leq& \Braces{ c_{\star}\Brackets{\Norm{\mathcal{P}_i t_i}_{a(\Omega\setminus K_i^{m-1})}^2+\Norm{\pi \mathcal{P}_i t_i}_{s(\Omega\setminus K_i^{m-1})}^2}}^{1/2}\Braces{\Norm{z_i}_a^2+\Norm{\pi z_i}_s^2}^{1/2}.
\end{aligned}
\]
Combining the result in \cref{lem:ab lem 1}, we hence complete this proof.
\end{proof}

\begin{proof}[Proof of \cref{lem:ab lem 3}]
Still take $z_i\coloneqq \Brackets{\mathcal{P}_i-\mathcal{P}_i^m}t_i$ and $z=\sum_{i=1}^N z_i$, and decompose $z$ as
\[
z=\Braces{\Brackets{1-\chi_i^{m,m+1}}z}+\Braces{\chi_i^{m,m+1}z}\coloneqq z'+z''.
\]
Noting that $\Supp\Brackets{z'} \subset \Omega\setminus K_i^{m}$, $\Supp \Brackets{\pi z'} \subset \Omega\setminus K_i^{m}$, $\Supp\Brackets{\mathcal{P}_i^m t_i} \subset \Cl\Brackets{K_i^m}$ and $\Supp\Brackets{\pi \mathcal{P}_i^m t_i} \subset \Cl\Brackets{K_i^m}$, we have
\[
a(\mathcal{P}_i^m t_i, z')+s(\pi\mathcal{P}_i^m t_i, \pi z')=0
\]
and
\[
a(\mathcal{P}_i t_i, z')+s(\pi\mathcal{P}_i t_i, \pi z')=\FuncAction{t_i}{z'}=0,
\]
which leads to
\[
a(z_i,z')+s(\pi z_i, \pi z')=0.
\]
Use similar techniques in the proof of \cref{lem:ab lem 2}:
\[
a(z_i, z)+s(\pi z_i, \pi z)=a(z_i, z'')+s(\pi z_i, \pi z'') \leq \Norm{z_i}_a\Norm{z''}_a+\Norm{\pi z_i}_s\Norm{\pi z''}_s;
\]
\[
\begin{aligned}
\Norm{z''}_a&=\Norm{\chi_i^{m,m+1}z}_a\leq \Norm{z}_{a\Brackets{K_i^{m+1}}}+\Norm{z}_{s\Brackets{K_i^{m+1}}}\\
&\leq \Brackets{1+\frac{1}{\sqrt{\Lambda}}}\Norm{z}_{a\Brackets{K_i^{m+1}}}+\Norm{\pi z}_{s\Brackets{K_i^{m+1}}};
\end{aligned}
\]
\[
\begin{aligned}
\Norm{\pi z''}_s &= \Norm{\pi \Brackets{\chi_i^{m,m+1}z}}_s \leq \Norm{\chi_i^{m,m+1}z}_s \leq \Norm{z}_{K_i^{m+1}} \\
&\leq \frac{1}{\sqrt{\Lambda}}\Norm{z}_{a\Brackets{K_i^{m+1}}}+\Norm{\pi z}_{s\Brackets{K_i^{m+1}}}.
\end{aligned}
\]
We hence obtain
\[
a(z_i,z)+s(\pi z_i, \pi z)\leq \Braces{c_\star\Brackets{\Norm{z}_{a\Brackets{K_i^{m+1}}}^2+\Norm{\pi z}_{s\Brackets{K_i^{m+1}}}^2}}^{1/2}\Braces{\Norm{z_i}_a^2+\Norm{\pi z_i}_s^2}^{1/2}.
\]
Recalling the definition of $C_\textnormal{ol}$, it is easy to show
\[
\sum_{i=1}^N \Norm{z}_{a\Brackets{K_i^{m+1}}}^2+\Norm{\pi z}_{s\Brackets{K_i^{m+1}}}^2\leq C_\textnormal{ol}\Brackets{m+1}^d\Brackets{\Norm{z}_a^2+\Norm{\pi z}_s^2}.
\]
Then by the Cauchy–Schwarz inequality, we get
\[
\begin{aligned}
\Norm{z}_a^2+\Norm{\pi z}_s^2&=\sum_{i=1}^N a(z_i,z)+s(\pi z_i,\pi z) \\
&\leq \Braces{c_\star C_\textnormal{ol} \Brackets{\Norm{z}_a^2+\Norm{\pi z}_s^2}}^{1/2}\Braces{\sum_{i=1}^N \Norm{z_i}_a^2+\Norm{\pi z_i}_s^2}^{1/2}.
\end{aligned}
\]
Combining \cref{lem:ab lem 2} and
\[
\Norm{\mathcal{P}_i t_i}_a^2+\Norm{\pi \mathcal{P}_i t_i}_s^2=\FuncAction{t_i}{\mathcal{P}_i t_i}
\]
prove the lemma.
\end{proof}

A direct result of \cref{lem:ab lem 3} is the following corollary, which presents estimates of $\Brackets{\mathcal{D}^\textnormal{glo}-\mathcal{D}^m}\tilde{g}$ and $\Brackets{\mathcal{N}^\textnormal{glo}-\mathcal{N}^m}q$.
\begin{corollary}\label{cor:diri numn esti}
Let notations be the same as \cref{lem:ab lem 1,lem:ab lem 2,lem:ab lem 3}. Then
\begin{equation}
\Norm{\Brackets{\mathcal{D}^\textnormal{glo}-\mathcal{D}^m)\tilde{g}}}_a^2+\Norm{\pi\Brackets{\mathcal{D}^\textnormal{glo}-\mathcal{D}^m)\tilde{g}}}_s^2\leq c_\star^2 C_\textnormal{ol}\theta^{m-1}\Brackets{m+1}^d\Norm{\tilde{g}}_a^2;
\end{equation}
and
\begin{equation}
\Norm{\Brackets{\mathcal{N}^\textnormal{glo}-\mathcal{N}^m)q}}_a^2+\Norm{\pi\Brackets{\mathcal{N}^\textnormal{glo}-\mathcal{N}^m)q}}_s^2\leq c_\star^2 C_\textnormal{ol}C_\textnormal{tr}^2\theta^{m-1}\Brackets{m+1}^d\Norm{q}^2_{L^2(\GammaN)},
\end{equation}
where $C_\textnormal{tr}$ is the modified norm of the trace operator $V\rightarrow L^2(\GammaN)$
\[
C_\textnormal{tr}\coloneqq \sup_{\substack{v\in V\\v\neq 0}} \frac{\Norm{v}_{L^2(\GammaN)}}{\Norm{v}_{a}}.
\]
\end{corollary}
\begin{proof}
According to \cref{lem:ab lem 3}, we are left to estimate
\[
\sum_{i=1}^N \int_{K_i} \Matrix{A}\nabla\tilde{g}\cdot \nabla \mathcal{D}^\textnormal{glo}_i \tilde{g} \dx x \quad\text{and}\quad \sum_{i=1}^N \int_{\partial K_i \cap \GammaN} q \mathcal{N}^\textnormal{glo}_i q \dx \sigma.
\]
Replacing $v$ with $\mathcal{D}^\textnormal{glo}_i \tilde{g}$ in \cref{eq:glob diri}, we have
\[
\Norm{\mathcal{D}^\textnormal{glo}_i\tilde{g}}_a^2 +\Norm{\pi \mathcal{D}^\textnormal{glo}_i\tilde{g}}_s^2 \leq \Norm{\tilde{g}}_{a(K_i)} \Norm{\mathcal{D}^\textnormal{glo}_i\tilde{g}}_{a(k_i)},
\]
which leads $\Norm{\mathcal{D}^\textnormal{glo}_i\tilde{g}}_{a(K_i)}\leq \Norm{\mathcal{D}^\textnormal{glo}_i\tilde{g}}_a\leq \Norm{\tilde{g}}_{a(K_i)}$ and
\[
\sum_{i=1}^N \int_{K_i} \Matrix{A}\nabla\tilde{g}\cdot \nabla \mathcal{D}^\textnormal{glo}_i \tilde{g} \dx x\leq \sum_{i=1}^N \Norm{\tilde{g}}_{a(K_i)}^2=\Norm{\tilde{g}}_a^2.
\]
Similarly, we can obtain $\Norm{\mathcal{N}^\textnormal{glo}_iq}_a\leq C_\textnormal{tr}\Norm{q}_{L^2\Brackets{\partial K_i \cap \GammaN}}$ from \cref{eq:glob neum} and
\[
\begin{aligned}
&\sum_{i=1}^N \int_{\partial K_i \cap \GammaN} q \mathcal{N}^\textnormal{glo}_i q \dx \sigma\leq \sum_{i=1}^N \Norm{q}_{L^2\Brackets{\partial K_i \cap \GammaN}} \Norm{\mathcal{N}^\textnormal{glo}_iq}_{L^2(\GammaN)} \\
\leq& \sum_{i=1}^N C_\textnormal{tr}\Norm{q}_{L^2\Brackets{\partial K_i \cap \GammaN}} \Norm{\mathcal{N}^\textnormal{glo}_iq}_{a}\leq \sum_{i=1}^N C_\textnormal{tr}^2 \Norm{q}^2_{L^2\Brackets{\partial K_i \cap \GammaN}}\\
=&C_\textnormal{tr} \Norm{q}_{L^2(\GammaN)}^2.
\end{aligned}
\]
\end{proof}

To systematically analyze function spaces $V^\textnormal{glo}_\textnormal{ms}$ and $V^m_\textnormal{ms}$, we also introduce another pair of operators $\mathcal{R}^\textnormal{glo}\coloneqq \sum_{i=1}^{N}\mathcal{R}^\textnormal{glo}_i:L^2(\Omega)\rightarrow V^\textnormal{glo}_\textnormal{ms}$ and $\mathcal{R}^m\coloneqq \sum_{i=1}^N \mathcal{R}_i^m: L^2(\Omega)\rightarrow V^m_\textnormal{ms}$, where $\mathcal{R}^\textnormal{glo}_i$ and $\mathcal{R}_i^m$ are defined as follows:
\begin{equation} \label{eq:glob R}
a(\mathcal{R}^\textnormal{glo}_i \varphi, v)+s(\pi\mathcal{R}^\textnormal{glo}_i \varphi, \pi v)=s(\pi_i \varphi, \pi v)\quad \forall v \in V;
\end{equation}
\begin{equation}\label{eq:local R}
a(\mathcal{R}_i^m \varphi, v)+s(\pi\mathcal{R}_i^m \varphi, \pi v)=s(\pi_i \varphi, \pi v)\quad \forall v \in V_i^m.
\end{equation}

The next lemma could be viewed as an inverse inequality of $V^\textnormal{aux}$ to $V$.
\begin{lemma}[\cite{Chung2018}] \label{lem:inverse}
There exists a positive constant $C_\textnormal{inv}$ such that for any $v\in L^2(\Omega)$, there exists $\hat{v}\in V$ with $\pi \hat{v}=\pi v$ and $\Norm{\hat{v}}_a\leq C_\textnormal{inv}\Norm{\pi v}_s$.
\end{lemma}
\begin{remark}
The existence of $C_\textnormal{inv}$ can be proven by the well-posedness of the corresponding saddle point problem \cite{Brezzi1991}. Moreover, as shown in \cite{Chung2018}, it is possible to provide a computable priori estimate of $C_\textnormal{inv}$ as
\[
C_\textnormal{inv} \leq 2\sup_{K_i\in\mathcal{T}^H}C_{\pi_i}\Brackets{\lambda_i^{l_i}+\frac{N_\textnormal{v}}{N_\textnormal{v}-1}},
\]
where
\[
C_{\pi_i}=\sup_{\substack{\mu\in V_i^\textnormal{aux}\\ \mu\neq 0}}\frac{\int_{K_i} \tilde{\kappa} \mu^2 \dx x}{\int_{K_i} \mathrm{B}(x) \tilde{\kappa} \mu^2 \dx x},
\]
and $\mathrm{B}(x) \in H^1_0(K_i)$ is a bubble function which can be expressed as the product of Lagrange bases on $K_i$ (i.e., $\mathrm{B}=\prod_{j=1}^{N_\textnormal{v}} \eta_i^j$).
\end{remark}

The main result of this subsection is the following theorem:
\begin{theorem}\label{thm:main}
Let $\mathcal{D}^m\tilde{g}$, $\mathcal{N}^m q$ and $w^m$ be the numerical solutions obtained in \textbf{Step1-4}, $w^\textnormal{glo}$ be the solution of \cref{eq:glob w}, and $\Lambda$, $\theta$, $c_\star$, $C_\textnormal{tr}$ and $C_\textnormal{inv}$ be the constants defined in \cref{thm:glob error}, \cref{lem:ab lem 1}, \cref{lem:ab lem 2}, \cref{cor:diri numn esti} and \cref{lem:inverse} respectively. Then
\begin{equation}
\begin{aligned}
&\Norm{w^m-\mathcal{D}^mg+\mathcal{N}^mq+\tilde{g}-u}_a \leq \frac{1}{\sqrt{\Lambda}} \Norm{f}_{s^{-1}}\\
&\qquad +c_\star \sqrt{C_\textnormal{ol}}\theta^{\frac{m-1}{2}}\Brackets{m+1}^{\frac{d}{2}}\Braces{\Norm{\tilde{g}}_a+C_\textnormal{tr}\Norm{q}_{L^2(\GammaN)}+C_\textnormal{inv}\Norm{w^\textnormal{glo}}_a+\Norm{\pi w^\textnormal{glo}}_s}.
\end{aligned}
\end{equation}
\end{theorem}

\begin{proof}
Take $e^m=w^m-\mathcal{D}^mg+\mathcal{N}^mq+\tilde{g}-u$. According to variational forms \cref{eq:model problem varia,eq:main solver}, we have $a(e^m, v)=0$ for all $v\in V^m_\textnormal{ms}$. Then for a $w^m_*\in V^m_\textnormal{ms}$ which will be chosen later, by the Galerkin orthogonality, it is easy to show
\[
\Norm{e^m}_a^2=\Norm{e^m-w^m+w^m_*}_a^2-\Norm{w^m-w^m_*}_a^2\leq \Norm{w^m_*-\mathcal{D}^mg+\mathcal{N}^mq+\tilde{g}-u}_a^2.
\]
Splitting $\Norm{w^m_*-\mathcal{D}^mg+\mathcal{N}^mq+\tilde{g}-u}_a$ into four terms leads to an estimate
\[
\begin{aligned}
&\Norm{w^m_*-\mathcal{D}^mg+\mathcal{N}^mq+\tilde{g}-u}_a \\
\leq& \Norm{w^\textnormal{glo}-\mathcal{D}^\textnormal{glo}\tilde{g}+\mathcal{N}^\textnormal{glo}q+\tilde{g}-u}+\Norm{\Brackets{\mathcal{D}^\textnormal{glo}-\mathcal{D}^m}\tilde{g}}_a+\Norm{\Brackets{\mathcal{N}^\textnormal{glo}-\mathcal{N}^m}q}_a+\Norm{w^\textnormal{glo}-w^m_*}_a\\
\leq& \frac{1}{\sqrt{\Lambda}}\Norm{f}_{s^{-1}}+c_\star\sqrt{C_\textnormal{ol}}\theta^{\frac{m-1}{2}}\Brackets{m+1}^{\frac{d}{2}}\Brackets{\Norm{\tilde{g}}_a+C_\textnormal{tr}\Norm{q}_{L^2(\GammaN)}}+\Norm{w^\textnormal{glo}-w^m_*}_a,
\end{aligned}
\]
where \cref{thm:glob error,cor:diri numn esti} are applied in the last line above. We are now left to estimate $\Norm{w^\textnormal{glo}-w^m_*}_a$.

The definition of $\mathcal{R}^\textnormal{glo}_i$ in \ref{eq:glob R} implies that $\mathcal{R}^\textnormal{glo}:L^2(\Omega)\rightarrow V^\textnormal{glo}_\textnormal{ms}$ is a surjective map. We are hence able to find $\varphi_*\in L^2(\Omega)$ such that $w^\textnormal{glo}=\mathcal{R}^\textnormal{glo} \varphi_*$. Meanwhile, setting $w_*^m=\mathcal{R}^m\varphi_*$, and the estimate of $\Norm{w^\textnormal{glo}-w^m_*}_a$ is transformed into $\Norm{\Brackets{\mathcal{R}^\textnormal{glo}-\mathcal{R}^m}\varphi_*}_a$. Utilizing similar techniques in proving \cref{cor:diri numn esti}, we can obtain
\[
\Norm{\Brackets{\mathcal{R}^\textnormal{glo}-\mathcal{R}^m}\varphi_*}_a^2+\Norm{\pi\Brackets{\mathcal{R}^\textnormal{glo}-\mathcal{R}^m}\varphi_*}_s^2\leq c_\star^2C_\textnormal{ol}\theta^{m-1}\Brackets{m+1}^d \Norm{\pi \varphi_*}_s^2.
\]
The variational form \ref{eq:glob R} yields a variational form for $\mathcal{R}^\textnormal{glo}$: for all $v\in V$,
\[
a(w^\textnormal{glo}, v)+s(\pi w^\textnormal{glo}, \pi v)=s(\pi \varphi_*, \pi v).
\]
According to \cref{lem:inverse}, it is possible to find $\hat{\varphi}_*\in V$ such that $\pi \hat{\varphi}_*=\pi \varphi_*$ and $\Norm{\hat{\varphi}_*}_a\leq C_\textnormal{inv}\Norm{\pi \varphi_*}_s$. Replacing $v$ with $\hat{\varphi}_*$ in the above variational form, we can obtain
\[
\begin{aligned}
\Norm{\pi \varphi_*}_s^2 &= a(w^\textnormal{glo},\hat{\varphi}_*)+s(\pi w^\textnormal{glo},\pi \varphi_*)\leq \Norm{\pi \varphi_*}_s\Brackets{C_\textnormal{inv}\Norm{w^\textnormal{glo}}_a+\Norm{\pi w^\textnormal{glo}}_s}\\
\Rightarrow \Norm{w^\textnormal{glo}-w^m_*}_a &\leq c_\star\sqrt{C_\textnormal{ol}}\theta^{\frac{m-1}{2}}\Brackets{m+1}^{\frac{d}{2}}\Brackets{C_\textnormal{inv}\Norm{w^\textnormal{glo}}_a+\Norm{\pi w^\textnormal{glo}}_s},
\end{aligned}
\]
and this finishes the proof.
\end{proof}

We can dig more estimates from this theorem. It is easy to see that $\Norm{f}_{s^{-1}} = O(H)$. Recalling \cref{eq:glob diri,eq:glob neum}, we have $\Norm{\mathcal{D}^\textnormal{glo}\tilde{g}}_a+\Norm{\pi\mathcal{D}^\textnormal{glo}\tilde{g}}_s=O(1)$ and $\Norm{\mathcal{N}^\textnormal{glo}q}_a+\Norm{\pi\mathcal{N}^\textnormal{glo}q}_s=O(1)$. From the proof of \cref{thm:glob error}, it holds that
\[
\begin{aligned}
\Norm{\pi w^\textnormal{glo}}_s&=\Norm{\pi \Brackets{u_0+\mathcal{D}^\textnormal{glo}\tilde{g}-\mathcal{N}^\textnormal{glo}q}}_s \leq \Norm{\pi u_0}_s + \Norm{\pi \mathcal{D}^\textnormal{glo}\tilde{g}}_s + \Norm{\pi \mathcal{N}^\textnormal{glo}q}_s \\
& \leq O(H^{-1}) + O(1).
\end{aligned}
\]
Meanwhile, by \cref{eq:glob w}, it follows that $\Norm{w^\textnormal{glo}}_a=O(1)$. If \emph{assuming} that $C_\textnormal{inv}=O(1)$, by choosing $m$ such that $\theta^{\frac{m-1}{2}}(m+1)^{d/2}=O(H^2)$, we will have
\[
\Norm{w^m-\mathcal{D}^mg+\mathcal{N}^mq+\tilde{g}-u}_a = O(H).
\]

\subsection{Extensions}
In this subsection, we will extend the CEM-GMsFEM to inhomogeneous Robin BVPs, and the model problem is stated as follows:
\[
\left\{
\begin{aligned}
&-\Div \Brackets{\Matrix{A}\nabla u} = f \ &\text{in}\  \Omega, \\
&\Matrix{b} u+\nu\cdot \Matrix{A}\nabla u = q \ &\text{on}\ \partial \Omega, \\
\end{aligned}
\right.
\]
where $\Matrix{b}(x) \in L^\infty(\partial \Omega)$ is a heterogeneous coefficient depends on certain physical laws. We propose an assumption for $\Matrix{b}$ to make this problem uniquely solvable in $V$:

\paragraph{A4} The function $\Matrix{b}(x)\geq 0$ for a.e. $x\in \partial\Omega$, and there exists a positive constant $b_0>0$ and a subset $\Gamma \subset \partial \Omega$ with $\Meas(\Gamma)>0$, such that $\Matrix{b}(x)\geq b_0$ for a.e. $x\in \Gamma$.

The bilinear form $a(w,v)$ needs to be modified as $\int_\Omega \Matrix{A}\nabla w\cdot \nabla v\dx x+\int_{\partial \Omega} \Matrix{b}wv\dx \sigma$, and for a subset $\omega \subset \Omega$ the norm $\Norm{\cdot}_{a(\omega)}$ is also redefined accordingly. Meanwhile, the eigenvalue problem for constructing the auxiliary space $V_i^\text{aux}$ will change into
\[
\int_{K_i}\Matrix{A}\nabla \phi_i\cdot \nabla v\dx x+\int_{\partial \Omega \cap \partial K_i} \Matrix{b}\phi_i v\dx \sigma=\lambda_i \int_{K_i} \tilde{\kappa} \phi_i v\dx x.
\]

The computational method consists of three steps:
\paragraph{Step1} Find $\mathcal{N}_i^m q\in V_i^m$ such that for all $v\in V_i^m$,
\[
a(\mathcal{N}_i^m q,v)+s(\pi \mathcal{N}_i^m q,\pi v)=\int_{\partial K_i \cap \partial \Omega} q v\dx \sigma\quad \forall v\in V_i^m.
\]
Then obtain $\mathcal{N}^mq=\sum_{i=1}^N \mathcal{N}_i^m q$.

\paragraph{Step2} Prepare the multiscale function space $V^m_\textnormal{ms}$ via \cref{eq:local basis varia} with a modified bilinear form $a(\cdot, \cdot)$.

\paragraph{Step3} Solve $w^m \in V^m_\textnormal{ms}$ such that for all $v \in V^m_\textnormal{ms}$,
\[
a(w^m, v)=\int_\Omega f v\dx x+\int_\GammaN q v\dx \sigma-a(\mathcal{N}^m q,v).
\]

\paragraph{Step4} Construct the numerical solution to approximate the real solution as
\[
u \approx w^m+\mathcal{N}^m q.
\]

The detailed numerial analysis is similar with \cref{subsec:analysis}, and we hence omit it here.

\section{Numerical experiments}\label{sec:experiment}

In this section, we will present several numerical experiments \footnote{All the experiments are conducted by using Python with Numpy \cite{Harris2020} and Scipy \cite{Virtanen2020} libraries, codes are hosted on Github (https://github.com/Laphet/CEM-GMsFEM.git).} to emphasize that the method proposed can retain accuracy in a high contrast coefficient setting. For simplicity, we take the domain $\Omega=(0, 1)\times (0, 1)$ and pointwise isotropic coefficients, i.e., $\Matrix{A}(x)=\kappa(x)\Matrix{I}$. In all experiments, the medium has two phases, which means $\kappa(x)$ only takes two values $\kappa_\textnormal{m}$ and $\kappa_\textnormal{I}$ with $1=\kappa_\textnormal{m} \ll \kappa_\textnormal{I}$. We will calculate reference solutions on a $400\times 400$ mesh with the bilinear Lagrange FEM, and $\kappa(x)$ is also generated from $400$px $\times 400$px figures. We display two different medium configurations in \cref{fig:media} and denote the first one by \textbf{cfg-a}, the second one by \textbf{cfg-b}. The coarse mesh size $H$ will be chosen from $\frac{1}{10}$, $\frac{1}{20}$, $\frac{1}{40}$ and $\frac{1}{80}$. For simplifying the implementation, We specially choose $\tilde{\kappa}=24 \kappa/H^2$ instead of the original definition \cref{eq:kappa}, and all theoretical results should still hold since we only require $\tilde{\kappa}$ satisfying $A\nabla \chi \cdot \nabla \chi \leq \tilde{\kappa}$ for any partition function $\chi$. Moreover, we always set $l_1=l_2=\cdots=l_N=l_\textnormal{m}$, i.e., the number of eigenvectors used to construct auxiliary space $V_i^\textnormal{aux}$ is fixed as $l_\textnormal{m}+1$.

\begin{figure}[tbhp]
\centering
\begin{subfigure}[tbhp]{0.495\textwidth}
\centering
\includegraphics[width=\linewidth]{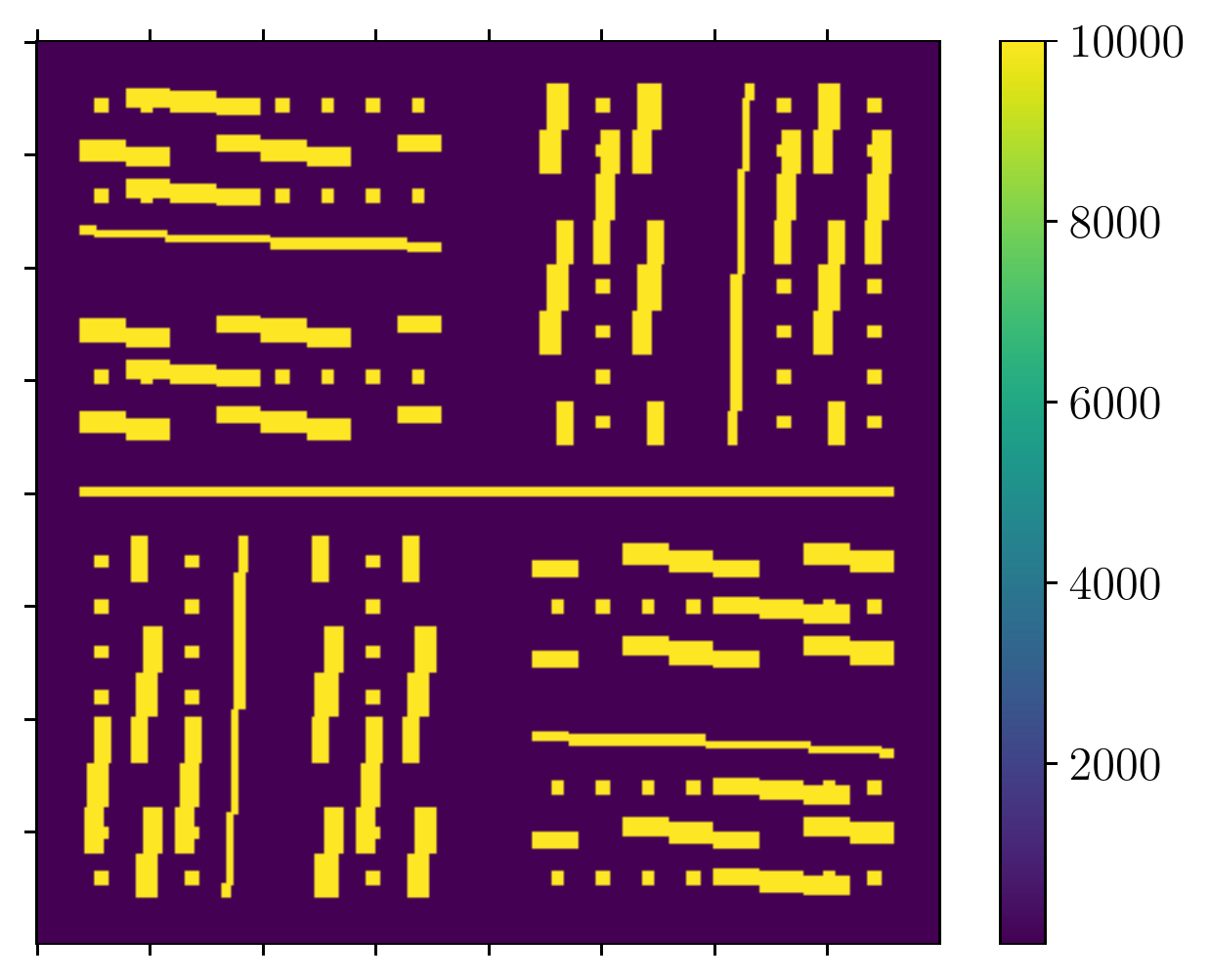}
\caption{}
\end{subfigure}
\begin{subfigure}[tbhp]{0.495\textwidth}
\centering
\includegraphics[width=\linewidth]{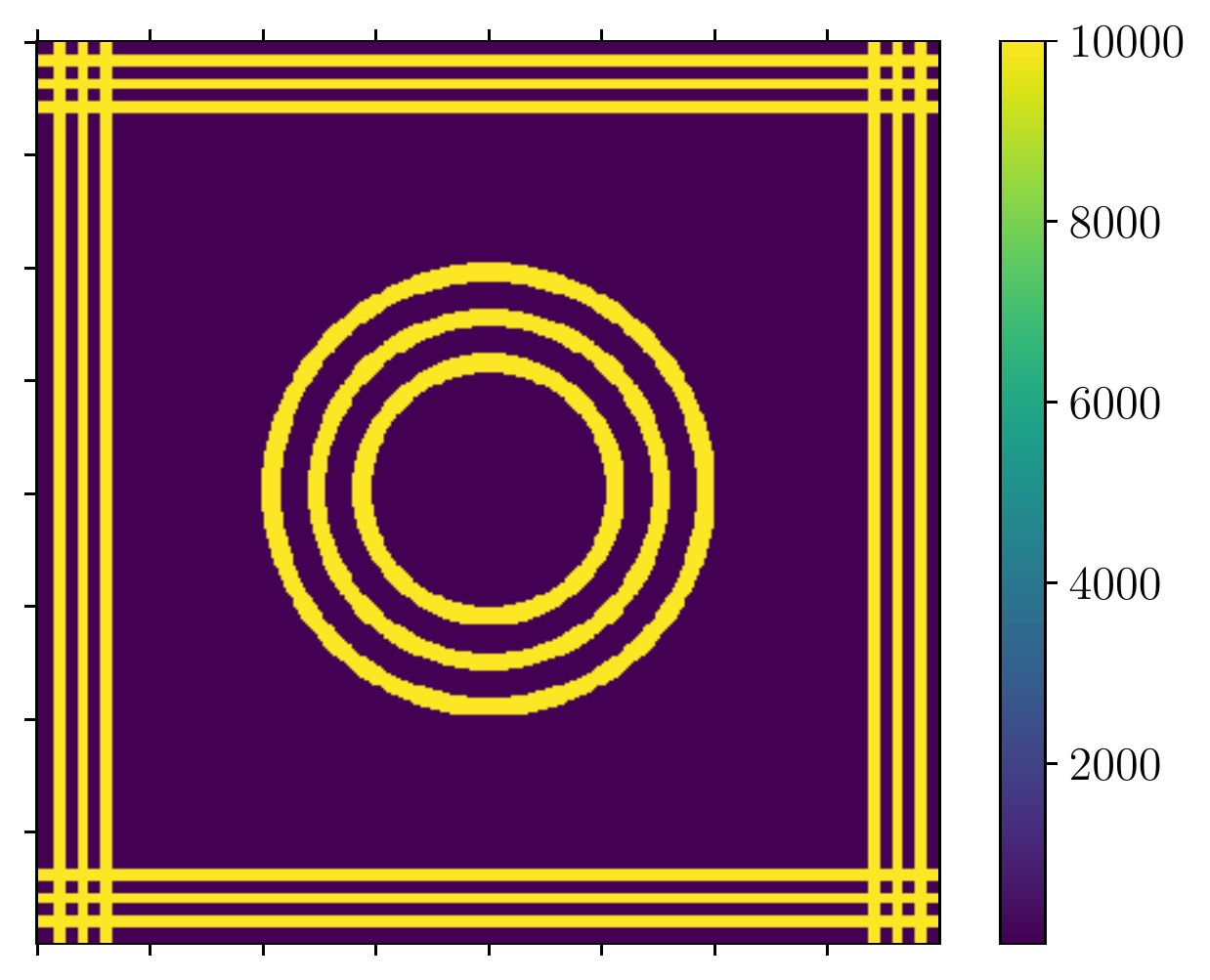}
\caption{}
\end{subfigure}
\begin{subfigure}[tbhp]{0.495\textwidth}
\centering
\includegraphics[width=\linewidth]{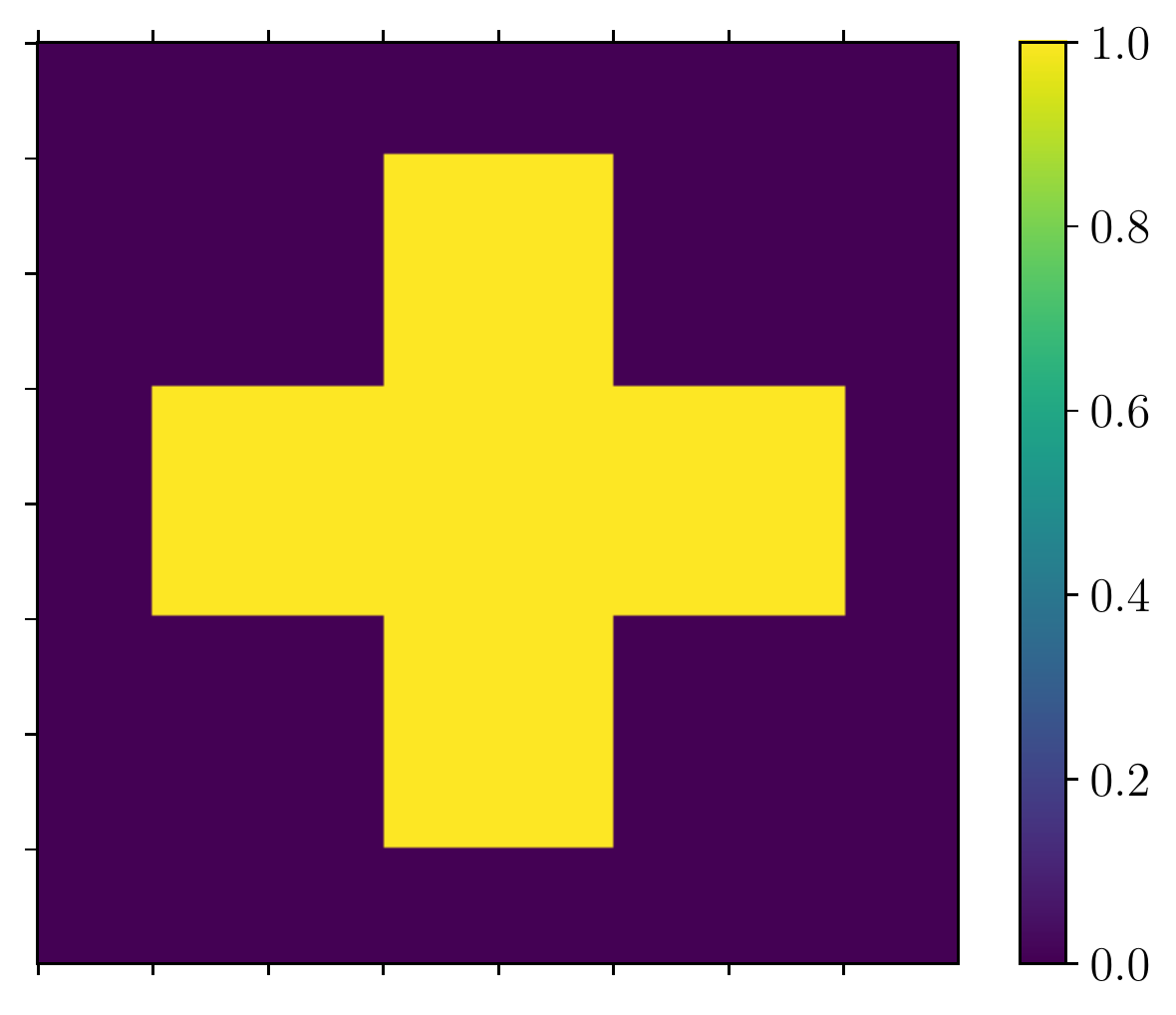}
\caption{}
\end{subfigure}
\caption{(a) The medium configuration \textbf{cfg-a}; (b) the medium configuration \textbf{cfg-b}; (c) the source term $f$.}\label{fig:media}
\end{figure}

\subsection{Model problem 1} \label{subsec:model1}
We consider the following model problem:
\begin{equation}\label{eq:Model1}
\left\{
\begin{aligned}
&-\Div\Brackets{\kappa(x_1, x_2)\nabla u} = f(x_1, x_2) \ &\forall (x_1, x_2)\in\Omega, \\
&u(x_1, x_2) = \tilde{g}(x_1, x_2)=x_1^2+\exp(x_1x_2) \  &\forall (x_1, x_2) \in \partial \Omega,
\end{aligned}
\right.
\end{equation}
where $\kappa$ is generated from \textbf{cfg-a} with various $\kappa_\textnormal{I}$ and the source term $f$ is a piecewise constant function whose value are taken via \cref{fig:media}.

We first examine the exponential convergence of $\Brackets{\mathcal{D}^m-\mathcal{D}^\textnormal{glo}}\tilde{g}$ by setting $H=1/20$ and $l_\textnormal{m}=2$. The results are reported in \cref{tab:Diri corr}, where we introduce notations
\[
\textnormal{D}_a^m \coloneqq \frac{\Norm{\Brackets{\mathcal{D}^m-\mathcal{D}^\textnormal{glo}}\tilde{g}}_a}{\Norm{\mathcal{D}^\textnormal{glo}\tilde{g}}_a} \quad \text{and} \quad \textnormal{D}_{L}^m \coloneqq \frac{\Norm{\Brackets{\mathcal{D}^m-\mathcal{D}^\textnormal{glo}}\tilde{g}}_{L^2(\Omega)}}{\Norm{\mathcal{D}^\textnormal{glo}\tilde{g}}_{L^2(\Omega)}}
\]
to measure errors, and $\Lambda'=\max_i \lambda_i^{l_\textnormal{m}}$ as a reference for $\Lambda$. We can see that the fluctuation of $\Lambda'$ with respect to the contrast ratio $\kappa_\textnormal{I}/\kappa_\textnormal{m}$ is almost unnoticeable in our test cases, which partly explains the effectiveness of the CEM-GMsFEM in high contrast problems. This observation is quite interest, and we can transform it into a formal mathematical conjecture: is it possible to bound $\lambda$ which is an eigenvalue of
\[
\int_\Omega \kappa \nabla u\cdot \nabla v\dx x = \lambda \int_\Omega \kappa u v\dx x \quad \forall v \in H^1(\Omega)
\]
 with $C\textnormal{diam}\Brackets{\Omega}^d$ while the constant $C$ is independent with $\sup_x \kappa/\inf_x \kappa$? As predicted in \cref{cor:diri numn esti}, the convergence behavior of $\textnormal{D}_a^m$ should solely depend on $\Lambda$, and our numerical results strongly support this argument. Comparing $\Norm{\mathcal{D}^\textnormal{glo}}_a$ with $\Norm{\mathcal{D}^\textnormal{glo}}_{L^2(\Omega)}$, we can find $\Norm{\mathcal{D}^\textnormal{glo}}_a$ is almost linearly dependent on contrast ratios while such a relation is not significant on $\Norm{\mathcal{D}^\textnormal{glo}}_{L^2(\Omega)}$. However, the exponential convergence property can allow us to compensate errors from $\Norm{\mathcal{D}^\textnormal{glo}}_a$ through a modest larger oversampling region.

\begin{table}[tbhp]
\centering
{\footnotesize
\caption{All the numerical tests are performed under $H=1/20$ and $l_\textnormal{m}=2$. The results include relative errors between $\mathcal{D}^\textnormal{glo}\tilde{g}$ and $\mathcal{D}^m\tilde{g}$ with respect to different oversampling layers $m$ and contrast ratios $\kappa_\textnormal{I}/\kappa_\textnormal{m}$, and the values of $\Lambda'=\max_i \lambda_i^{l_\textnormal{m}}$, $\Norm{\mathcal{D}^\textnormal{glo}\tilde{g}}_a$ and $\Norm{\mathcal{D}^\textnormal{glo}\tilde{g}}_{L^2(\Omega)}$ with different contrast ratios.}\label{tab:Diri corr}
\begin{tabular}{|c|c|c|c|}
\hline
 $\kappa_\textnormal{I}/\kappa_\textnormal{m}$ & \num{1.000e+4} & \num{1.000e+5} & \num{1.000e+6}  \\
\hline
$\Lambda'$ & \num{1.149e+00} & \num{1.149e+00} & \num{1.149e+00} \\
\hline
\hline
$\textnormal{D}_a^1$ & \num{1.052e-2} & \num{1.051e-2} & \num{1.051e-02} \\
\hline
$\textnormal{D}_a^2$ & \num{2.575e-4} & \num{2.568e-4} & \num{2.567e-04} \\
\hline
$\textnormal{D}_a^3$ & \num{6.679e-6} & \num{6.592e-6} & \num{6.583e-06} \\
\hline
$\textnormal{D}_a^4$ & $<$\num{1.000e-6} & $<$\num{1.000e-6} & $<$\num{1.000e-6} \\
\hline
$\Norm{\mathcal{D}^\textnormal{glo}}_a$ & \num{6.767e+01} & \num{2.140e+02} & \num{6.769e+02} \\

\hline
\hline
$\textnormal{D}_{L}^1$ & \num{3.944e-02} & \num{3.941e-02} & \num{3.941e-02} \\
\hline
$\textnormal{D}_{L}^2$ & \num{1.583e-03} & \num{1.582e-03} & \num{1.582e-03} \\
\hline
$\textnormal{D}_{L}^3$ & \num{1.440e-04} & \num{1.438e-04} & \num{1.438e-04} \\
\hline
$\textnormal{D}_{L}^4$ & $<$\num{1.000e-6} & $<$\num{1.000e-6} & $<$\num{1.000e-6} \\
\hline
$\Norm{\mathcal{D}^\textnormal{glo}}_{L^2(\Omega)}$ & \num{6.947e-03} & \num{6.952e-03} & \num{6.953e-03} \\
\hline
\end{tabular}
}
\end{table}

In the second experiment, we fix $\kappa_\textnormal{I}/\kappa_\textnormal{m}=10^4$ and $l_\text{m}=2$ to display how numerical solution errors, i.e.,
\[
\textnormal{E}_a^m\coloneqq \frac{\Norm{w^m-\mathcal{D}^m\tilde{g}-u_0}_a}{\Norm{u}_a} \quad \text{and}\quad \textnormal{E}_a^m\coloneqq \frac{\Norm{w^m-\mathcal{D}^m\tilde{g}-u_0}_{L^2(\Omega)}}{\Norm{u}_{L^2(\Omega)}}
\]
change with coarse mesh sizes $H$ and oversampling layers $m$. The results are reported in \cref{tab:Diri Hm}. An important observation is that errors will increase if we only reduce $H$ while not enlarge oversampling layers $m$, which is distinct from traditional finite element methods. Actually, if focusing on $m=1$, we can see that $\textnormal{E}_a^1$ grows like $O(H^{-1})$. Recalling the error expression in \cref{thm:main}, such a numerical evidence supports the estimate $\Norm{\pi w^\textnormal{glo}}_s=O(H^{-1})$. Due to the competence between $\Norm{\pi w^\textnormal{glo}}_s$ and $m$, the minimal error occurs in the setting ($H=1/40,m=4$) rather than ($H=1/80,m=4$). Another interesting point is that, when $m=1$, errors in the $L^2$ norm ($\approx 8\%$) are greatly smaller than ones in the energy norm ($70\%\sim 600\%$). Since $L^2$ norms can not reflect small oscillations of functions, we may image that by gradually increasing oversampling layers, numerical solutions obtained by the CEM-GMsFEM capture ``macroscale'' information first then resolve ``finescale'' details.

\begin{table}[tbhp]
\centering
{\footnotesize
\caption{The numerical errors of the model problem \cref{eq:Model1} in the  energy and $L^2$ norm with different coarse mesh sizes $H$ and oversampling layers $m$, while contrast ratios and eigenvector numbers are fixed as $\kappa_\textnormal{I}/\kappa_\textnormal{m}=10^4$ and $l_\textnormal{m}=2$.}\label{tab:Diri Hm}
\begin{tabular}{|c|c|c|c|c|}
\hline
$H$ & $1/10$ & $1/20$ & $1/40$ & $1/80$  \\
\hline
$\textnormal{E}_{a}^1$ & \num{7.702e-01} & \num{1.453e+00} & \num{3.065e+00} & \num{6.029e+00} \\
\hline
$\textnormal{E}_{a}^2$ & \num{4.023e-02} & \num{8.161e-02} & \num{2.005e-01} & \num{4.401e-01} \\
\hline
$\textnormal{E}_{a}^3$ & \num{2.662e-03} & \num{2.632e-03} & \num{7.753e-03} & \num{2.301e-02} \\
\hline
$\textnormal{E}_{a}^4$ & \num{2.308e-03} & \num{4.283e-04} & \num{3.041e-04} & \num{1.035e-03} \\
\hline
\hline
$\textnormal{E}_{L}^1$ & \num{6.957e-02} & \num{6.603e-02} & \num{7.445e-02} & \num{8.062e-02} \\
\hline
$\textnormal{E}_{L}^2$ & \num{6.789e-04} & \num{3.237e-03} & \num{1.664e-02} & \num{4.581e-02} \\
\hline
$\textnormal{E}_{L}^3$ & \num{7.070e-05} & \num{7.016e-06} & \num{2.860e-05} & \num{2.315e-04} \\
\hline
$\textnormal{E}_{L}^4$ & \num{6.638e-05} & \num{4.857e-06} & \num{1.079e-06} & \num{1.079e-06} \\
\hline
\end{tabular}
}
\end{table}

In the third experiment, we focus on numerical errors with different contrast ratios and oversampling layers, and record the results in \cref{tab:Diri ContrM}, where we set $H=1/80$ and $l_\textnormal{m}=2$. It is not surprise that high contrast ratios will deteriorate numerical accuracy. However, the exponential convergence in $m$ alleviates this deterioration. Actually, the numerical accuracy of ($\kappa_\textnormal{I}/\kappa_\textnormal{m}=10^6,m=4$) improves almost $20$ times comparing to ($\kappa_\textnormal{I}/\kappa_\textnormal{m}=10^6,m=3$). Similarly, as emphasized in the second experiment, the $L^2$ norm errors ($\approx 8\%$) are significantly smaller than energy norm errors ($190\%\sim 600\%$) when $m=1$. This phenomenon reveals the potential of the CEM-GMsFEM in discovering homogenized surrogate models of high contrast problems \cite{Chung2018a}.

\begin{table}[tbhp]
\centering
{\footnotesize
\caption{The numerical errors the model problem \cref{eq:Model1} in the  energy and $L^2$ norm with different contrast rations $\kappa_\textnormal{I}/\kappa_\textnormal{m}$ and oversampling layers $m$, while coarse mesh sizes and eigenvector numbers are fixed as $H=1/80$ and $l_\textnormal{m}=2$.} \label{tab:Diri ContrM}
\begin{tabular}{|c|c|c|c|c|}
\hline
$\kappa_\textnormal{I}/\kappa_\textnormal{m}$ & \num{1.000e+3} & \num{1.000e+4} & \num{1.000e+5} & \num{1.000e+6} \\
\hline
$\textnormal{E}_{a}^1$ & \num{1.944e+00} & \num{6.029e+00} & \num{1.902e+01} & \num{6.013e+01} \\
\hline
$\textnormal{E}_{a}^2$ & \num{1.790e-01} & \num{4.401e-01} & \num{1.061e+00} & \num{3.002e+00} \\
\hline
$\textnormal{E}_{a}^3$ & \num{7.882e-03} & \num{2.301e-02} & \num{7.097e-02} & \num{2.075e-01} \\
\hline
$\textnormal{E}_{a}^4$ & \num{3.943e-04} & \num{1.035e-03} & \num{3.141e-03} & \num{9.882e-03} \\
\hline
$\Norm{u}_a$ & \num{2.826e+00} & \num{2.841e+00} & \num{2.843e+00} & \num{2.843e+00} \\
\hline
\hline
$\textnormal{E}_{L}^1$ & \num{7.866e-02} & \num{8.062e-02} & \num{8.086e-02} & \num{8.089e-02} \\
\hline
$\textnormal{E}_{L}^2$ & \num{1.330e-02} & \num{4.581e-02} & \num{6.632e-02} & \num{7.250e-02} \\
\hline
$\textnormal{E}_{L}^3$ & \num{2.968e-05} & \num{2.315e-04} & \num{2.174e-03} & \num{1.657e-02} \\
\hline
$\textnormal{E}_{L}^4$ & \num{1.079e-06} & \num{1.079e-06} & \num{9.175e-06} & \num{1.241e-05} \\
\hline
$\Norm{u}_{L^2(\Omega)}$ & \num{1.853e+00} & \num{1.853e+00} & \num{1.853e+00} & \num{1.853e+00} \\
\hline
\end{tabular}
}
\end{table}

We test numerical errors with different $l_m$ in the fourth experiment, and the other parameters are set as ($H=1/80,\kappa_\textnormal{I}/\kappa_\textnormal{m}=10^6,m=3$). It is natural that providing more eigenvectors in constructing $V^\textnormal{aux}$ performance will be better. However, the numerical report \cref{tab:Diri eig} shows that error may not be reduced proportionally to eigenvector numbers. Because adding more eigenvectors improves convergence rates through increasing $\Lambda$, which we only know the asymptotic behavior when $l_\textnormal{m}\rightarrow \infty$ (i.e., Weyl's law \cite{Ivrii2016}). In practice, $3$ or $4$ eigenvectors is enough for obtaining satisfying accuracy, and there are several rules of thumb in determining how many eigenvectors should be applied when high conductivity channels appear \cite{Chung2018}.

\begin{table}[tbhp]
\centering
{\footnotesize
\caption{The numerical errors the model problem \cref{eq:Model1} in the  energy and $L^2$ norm with different numbers ($l_\textnormal{m}+1$) of eigenvectors applied in constructing $V^\textnormal{aux}_i$, while other parameters are fixed as ($H=1/80,\kappa_\textnormal{I}/\kappa_\textnormal{m}=10^6,m=3$).} \label{tab:Diri eig}
\begin{tabular}{|c|c|c|c|c|}
\hline
$l_\textnormal{m}$ & $0$ & $1$ & $2$ & $3$ \\
\hline
$\textnormal{E}_{a}^3$ & \num{8.002e-01} & \num{4.932e-01} & \num{2.301e-02} & \num{2.109e-02} \\
\hline
$\textnormal{E}_{L}^3$ & \num{6.297e-02} & \num{3.589e-02} & \num{2.315e-04} & \num{2.002e-04} \\
\hline
\end{tabular}
}
\end{table}

\subsection{Model problem 2}
In this subsection, we study the following inhomogeneous Neumann BVP:
\begin{equation} \label{eq:Model2}
\left\{
\begin{aligned}
&-\Div\Brackets{\kappa(x_1, x_2)\nabla u} = f(x_1, x_2) \ &\forall (x_1, x_2)\in\Omega, \\
&u(x_1, x_2) = 0 \  &\forall (x_1, x_2) \in (0, 1) \times \{1\}, \\
&\nu \cdot \kappa \nabla u = q(x_1, x_2) = -1 \ &\forall (x_1, x_2) \in \{0\} \times (0, 1), \\
&\nu \cdot \kappa \nabla u = q(x_1, x_2) = 1 \ &\forall (x_1, x_2) \in \{1\} \times (0, 1), \\
&\nu \cdot \kappa \nabla u = q(x_1, x_2) = 1 \ &\forall (x_1, x_2)  \in (0, 0.5) \times \{0\}, \\
&\nu \cdot \kappa \nabla u = q(x_1, x_2) = 0 \ &\forall (x_1, x_2)  \in (0.5, 1) \times \{0\}, \\
\end{aligned}
\right.
\end{equation}
where $\kappa$ is generated from \textbf{cfg-b} with various $\kappa_\textnormal{I}$ and the source term $f$ is a piecewise constant function whose value are taken via \cref{fig:media}.

We first plot $\mathcal{N}^\textnormal{glo}q$ under different contrast ratios ($\kappa_\textnormal{I}/\kappa_\textnormal{m}=1,10,10^2,10^3$) in \cref{fig:Neum corr}, and note that when $\kappa_\textnormal{I}/\kappa_\textnormal{m}=1$ the medium is homogeneous. Due to high conductivity channels contacting with boundary, we can see from \cref{fig:Neum corr} that on the four corners of $\Omega$, $\mathcal{N}^\textnormal{glo}q$ of the test case ($\kappa_\textnormal{I}/\kappa_\textnormal{m}=1$) on $\GammaN$ is different from rest cases. Moreover, all figures show that $\mathcal{N}^\textnormal{glo}q$ decays rapidly away from boundaries, which is already predicted by \cref{cor:diri numn esti}. We also collect several numerical results \cref{tab:Neum corr} to illustrate this point, where notations
\[
\textnormal{N}_a^m \coloneqq \frac{{\Norm{\Brackets{\mathcal{N}^m-\mathcal{N}^\textnormal{glo}}q}}_a}{\Norm{\mathcal{N}^\textnormal{glo}q}_a} \quad \text{and}  \quad \textnormal{N}_{L}^m \coloneqq \frac{\Norm{\Brackets{\mathcal{N}^m-\mathcal{N}^\textnormal{glo}}q}_{L^2(\Omega)}}{\Norm{\mathcal{D}^\textnormal{glo}q}_{L^2(\Omega)}}
\]
are adopted. We can read from \cref{tab:Neum corr} that changes of $\Lambda'$ is quite small, which has been observed in \cref{subsec:model1}. Owing to this fact, even we multiply by $10$ on contrast ratios column by column, convergence histories of $\textnormal{N}_a^m$ along with $m$ are still similar. The major difference with \cref{tab:Diri corr} is that $\Norm{\mathcal{N}^\textnormal{glo}}_a$ does not grow linearly with respect to contrast ratios. This can be explained by the estimate $\Norm{\mathcal{N}^\textnormal{glo}}_a \leq C_\textnormal{tr} \Norm{q}_{L^2(\GammaN)}$, and we may postulate that the influence of contrast ratios on $C_\textnormal{tr}$ is limited.

\begin{figure}[tbhp]
\centering
\begin{subfigure}[tbhp]{0.495\textwidth}
\centering
\includegraphics[width=\linewidth]{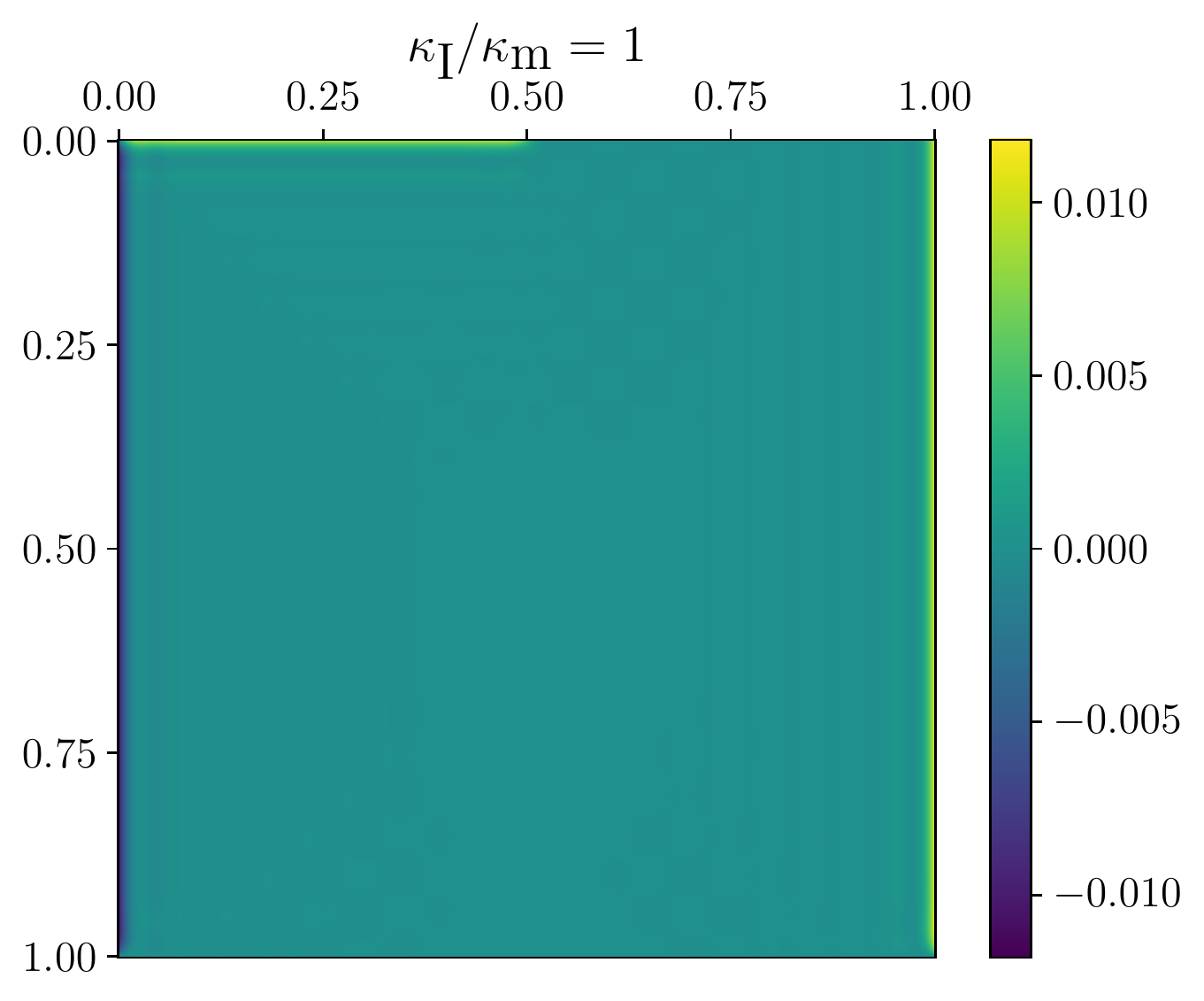}
\caption{}
\end{subfigure}
\begin{subfigure}[tbhp]{0.495\textwidth}
\centering
\includegraphics[width=\linewidth]{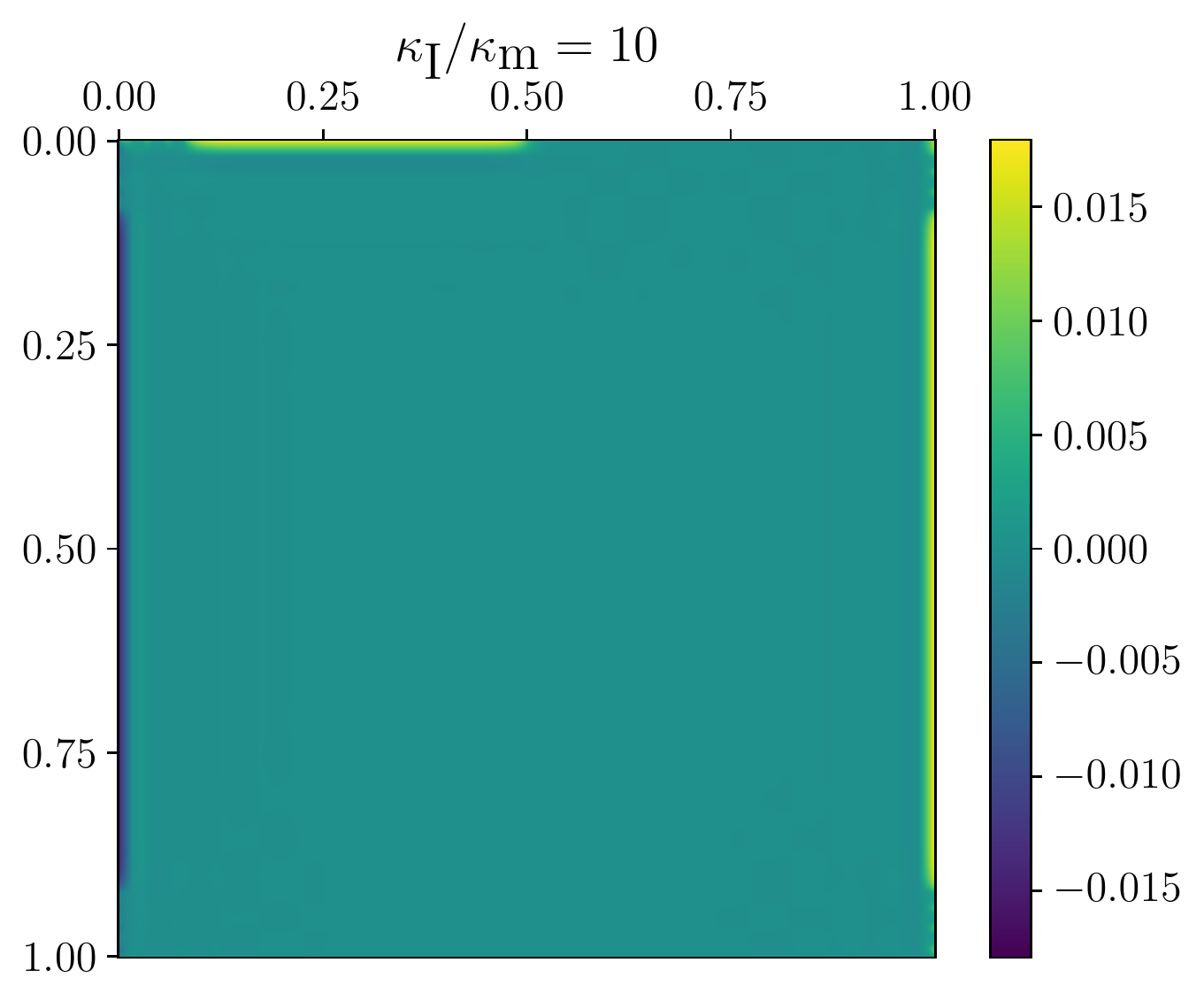}
\caption{}
\end{subfigure}
\begin{subfigure}[tbhp]{0.495\textwidth}
\centering
\includegraphics[width=\linewidth]{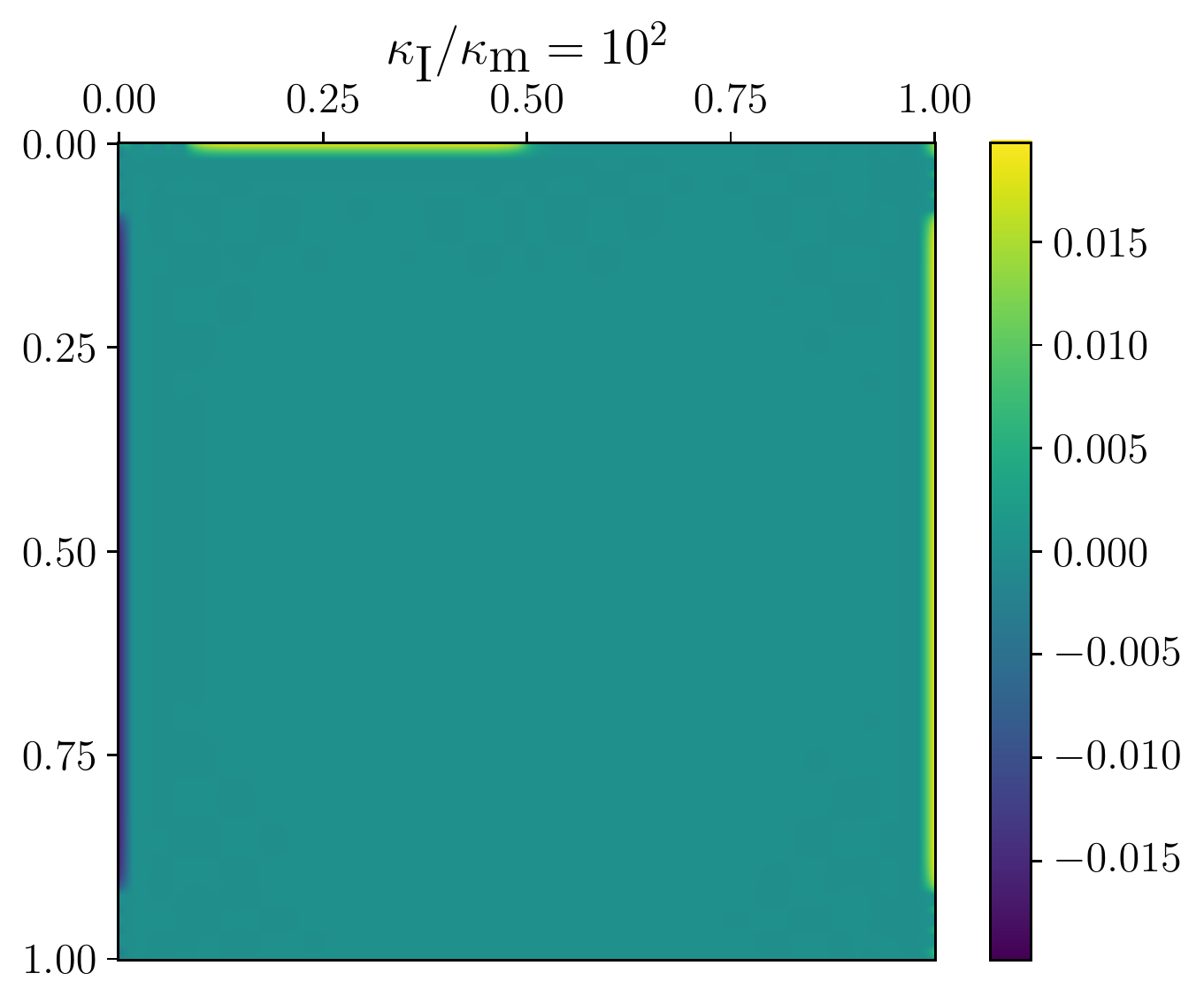}
\caption{}
\end{subfigure}
\begin{subfigure}[tbhp]{0.495\textwidth}
\centering
\includegraphics[width=\linewidth]{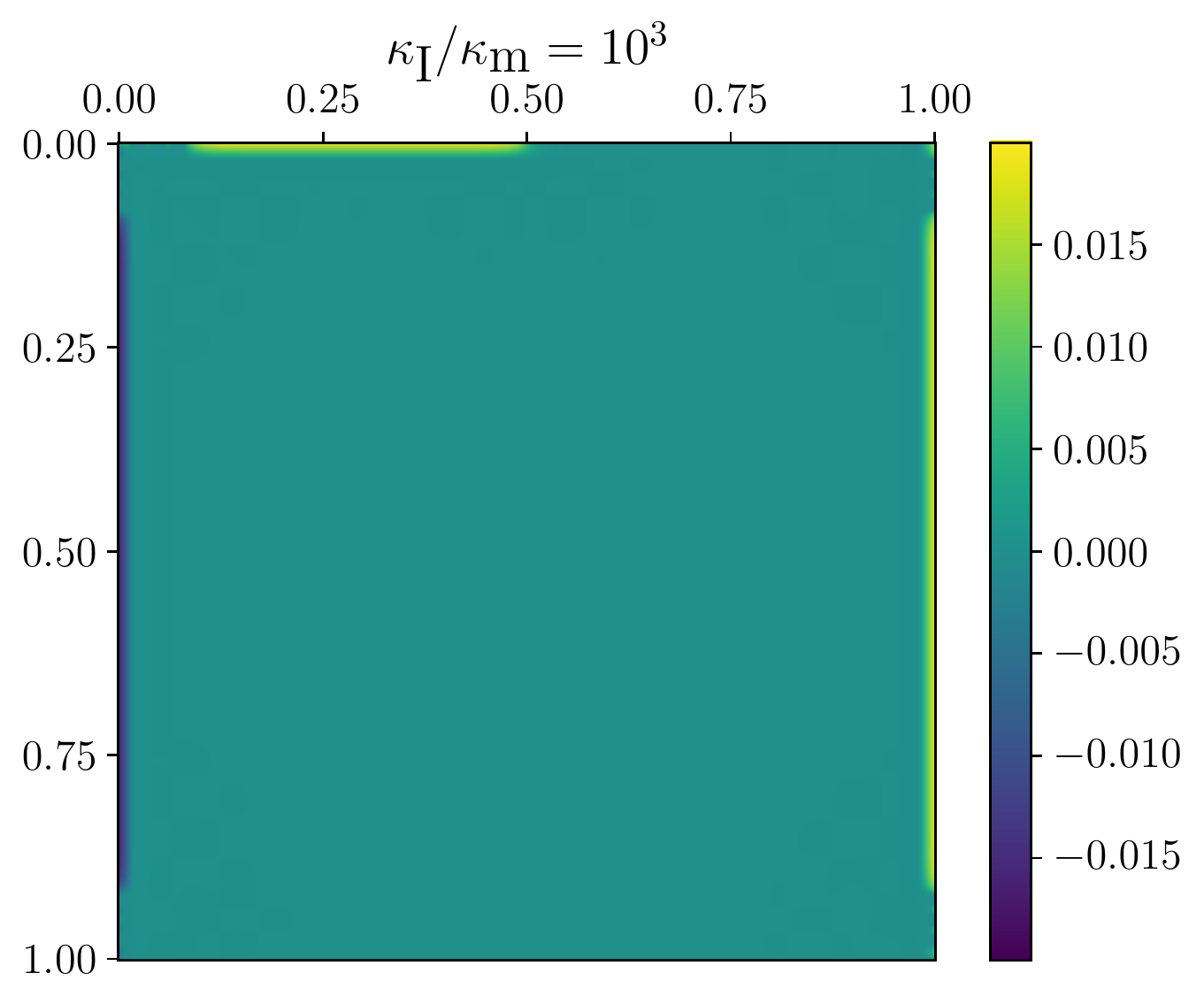}
\caption{}
\end{subfigure}
\caption{Plot $\mathcal{N}^\textnormal{glo}q$ with different contrast ratios. Note that (a) is distinct from (b) to (d) on the four corners of the domain.}\label{fig:Neum corr}
\end{figure}

\begin{table}[tbhp]
\centering
{\footnotesize
\caption{All the numerical tests are performed under $H=1/20$ and $l_\textnormal{m}=2$. The results include relative errors between $\mathcal{N}^\textnormal{glo}\tilde{g}$ and $\mathcal{N}^m\tilde{g}$ with respect to different oversampling layers $m$ and contrast ratios $\kappa_\textnormal{I}/\kappa_\textnormal{m}$, and the values of $\Lambda'=\max_i \lambda_i^{l_\textnormal{m}}$, $\Norm{\mathcal{N}^\textnormal{glo}q}_a$ and $\Norm{\mathcal{N}^\textnormal{glo}q}_{L^2(\Omega)}$ with different contrast ratios.} \label{tab:Neum corr}
\begin{tabular}{|c|c|c|c|c|c|}
\hline
$\kappa_\textnormal{I}/\kappa_\textnormal{m}$ &  \num{1.e+2} & \num{1.e+3} & \num{1.e+4} & \num{1.e+5} & \num{1.e+6}  \\
\hline
$\Lambda'$ &  \num{8.726e-01} & \num{8.802e-01} & \num{8.810e-01} & \num{8.811e-01} & \num{8.811e-01} \\
\hline
\hline
$\textnormal{N}_a^1$ & \num{9.941e-03} & \num{9.949e-03} & \num{9.949e-03} & \num{9.949e-03} & \num{9.949e-03} \\
\hline
$\textnormal{N}_a^2$  & \num{3.133e-04} & \num{1.911e-04} & \num{1.760e-04} & \num{1.709e-04} & \num{1.709e-04} \\
\hline
$\Norm{\mathcal{N}^\textnormal{glo}q}_a$  & \num{1.979e-01} & \num{1.988e-01} & \num{1.989e-01} & \num{1.989e-01} & \num{1.989e-01} \\
\hline
\hline
$\textnormal{N}_{L}^1$  & \num{8.127e-03} & \num{8.475e-03} & \num{8.467e-03} & \num{8.467e-03} & \num{8.467e-03} \\
\hline
$\textnormal{N}_L^2$  &  $<$\num{1.000e-6} &  $<$\num{1.000e-6} &  $<$\num{1.000e-6} &  $<$\num{1.000e-6} &  $<$\num{1.000e-6} \\
\hline
$\Norm{\mathcal{N}^\textnormal{glo}q}_{L^2(\Omega)}$ &  \num{2.338e-03} & \num{2.360e-03} & \num{2.362e-03} & \num{2.362e-03} & \num{2.362e-03} \\
\hline
\end{tabular}
}
\end{table}

The next experiment is parallel to the third experiment in \cref{subsec:model1}, and the main objective is verifying the effectiveness of the method proposed in high contrast settings. We choose $H=1/80$ and $l_\textnormal{m}=2$, and the results are listed in \cref{tab:Neum ContrM}, where a similar convergence pattern as \cref{tab:Diri ContrM} can be observed. We again emphasized that to achieve expected numerical accuracy, enlarging oversampling regions is necessary.

\begin{table}[tbhp]
\centering
{\footnotesize
\caption{The numerical errors of the model problem \cref{eq:Model2} in the energy and $L^2$ norm with different contrast rations $\kappa_\textnormal{I}/\kappa_\textnormal{m}$ and oversampling layers $m$, while coarse mesh sizes and eigenvector numbers are fixed as $H=1/80$ and $l_\textnormal{m}=2$.} \label{tab:Neum ContrM}
\begin{tabular}{|c|c|c|c|c|}
\hline
$\kappa_\textnormal{I}/\kappa_\textnormal{m}$ & \num{1.000e+3} & \num{1.000e+4} & \num{1.000e+5} & \num{1.000e+6} \\
\hline
$\textnormal{E}_{a}^1$ & \num{5.847e-01} & \num{4.263e-01} & \num{3.956e-01} & \num{3.922e-01} \\
\hline
$\textnormal{E}_{a}^2$ & \num{2.535e-01} & \num{3.499e-01} & \num{3.784e-01} & \num{3.820e-01} \\
\hline
$\textnormal{E}_{a}^3$ & \num{1.290e-02} & \num{3.155e-02} & \num{9.060e-02} & \num{2.321e-01} \\
\hline
$\textnormal{E}_{a}^4$ & \num{7.495e-04} & \num{1.415e-03} & \num{3.644e-03} & \num{1.117e-02} \\
\hline
$\Norm{u}_a$ & \num{2.508e-01} & \num{2.206e-01} & \num{2.168e-01} & \num{2.164e-01} \\
\hline
\hline
$\textnormal{E}_{L}^1$ & \num{9.582e-01} & \num{9.575e-01} & \num{9.578e-01} & \num{9.578e-01} \\
\hline
$\textnormal{E}_{L}^2$ & \num{3.579e-01} & \num{7.894e-01} & \num{9.280e-01} & \num{9.461e-01} \\
\hline
$\textnormal{E}_{L}^3$ & \num{9.784e-04} & \num{6.477e-03} & \num{5.321e-02} & \num{3.491e-01} \\
\hline
$\textnormal{E}_{L}^4$ & $<$\num{1.000e-6} & $<$\num{1.000e-6} & \num{6.489e-05} & \num{7.804e-04} \\
\hline
$\Norm{u}_{L^2(\Omega)}$ & \num{1.942e-02} & \num{1.575e-02} & \num{1.541e-02} & \num{1.538e-02} \\
\hline
\end{tabular}
}
\end{table}

\subsection{Model problem 3}
In this subsection, we consider the following inhomogeneous Robin BVP:
\begin{equation} \label{eq:Model3}
\left\{
\begin{aligned}
&-\Div\Brackets{\kappa(x_1, x_2)\nabla u} = f(x_1, x_2) \ &\forall (x_1, x_2)\in\Omega, \\
&\nu \cdot \kappa(x_1,x_2) \nabla u + \Matrix{b}(x_1,x_2) u= q(x_1, x_2) \ &\forall (x_1, x_2) \in \partial \Omega,
\end{aligned}
\right.
\end{equation}
where $\kappa(x_1,x_2)$ is generated from \textbf{cfg-b} with $\Matrix{b}(x_1,x_2)=\kappa(x_1,x_2)$, and $q(x_1, x_2)$ is defined as
\[
q(x_1,x_2)=\left\{
\begin{aligned}
&-1\quad \text{on}\quad \{0\}\times(0,1), \\
&1\quad \text{on}\quad \{1\}\times(0,1), \\
&1\quad \text{on}\quad (0,0.5)\times\{0\}, \quad 0\quad \text{on}\quad(0.5,1)\times\{0\},\\
&0\quad \text{on}\quad (0,0.5)\times\{1\}, \quad -1\quad \text{on}\quad(0.5,1)\times\{1\}.
\end{aligned}
\right.
\]

Note that $\Matrix{b}(x_1,x_2)$ is also high contrast now, and traditional methods need a very fine mesh to resolve the channel structure and special treatments to solve final linear systems. The numerical results of our method are listed in \cref{tab:Robin ContrM}, which shows that the reliability of our method even with a $10^6$ contrast ratio.

\begin{table}[tbhp]
\centering
{\footnotesize
\caption{The numerical errors of the model problem \cref{eq:Model3} in the  energy and $L^2$ norm with different contrast rations $\kappa_\textnormal{I}/\kappa_\textnormal{m}$ and oversampling layers $m$, while coarse mesh sizes and eigenvector numbers are fixed as $H=1/80$ and $l_\textnormal{m}=2$.} \label{tab:Robin ContrM}
\begin{tabular}{|c|c|c|c|c|}
\hline
$\kappa_\textnormal{I}/\kappa_\textnormal{m}$ & \num{1.000e+3} & \num{1.000e+4} & \num{1.000e+5} & \num{1.000e+6} \\
\hline
$\textnormal{E}_{a}^1$ & \num{5.293e-01} & \num{3.960e-01} & \num{3.711e-01} & \num{3.684e-01} \\
\hline
$\textnormal{E}_{a}^2$ & \num{2.135e-01} & \num{3.248e-01} & \num{3.544e-01} & \num{3.581e-01} \\
\hline
$\textnormal{E}_{a}^3$ & \num{1.091e-02} & \num{2.933e-02} & \num{8.487e-02} & \num{2.175e-01} \\
\hline
$\textnormal{E}_{a}^4$ & \num{6.657e-04} & \num{1.315e-03} & \num{3.414e-03} & \num{1.046e-02} \\
\hline
$\Norm{u}_a$ & \num{2.584e-01} & \num{2.342e-01} & \num{2.311e-01} & \num{2.308e-01} \\
\hline
\hline
$\textnormal{E}_{L}^1$ & \num{9.494e-01} & \num{9.555e-01} & \num{9.565e-01} & \num{9.565e-01} \\
\hline
$\textnormal{E}_{L}^2$ & \num{3.613e-01} & \num{7.899e-01} & \num{9.269e-01} & \num{9.449e-01} \\
\hline
$\textnormal{E}_{L}^3$ & \num{1.008e-03} & \num{6.499e-03} & \num{5.315e-02} & \num{3.487e-01} \\
\hline
$\textnormal{E}_{L}^4$ & $<$\num{1.000e-6} & $<$\num{1.000e-6} & \num{6.490e-05} & \num{7.795e-04} \\
\hline
$\Norm{u}_{L^2(\Omega)}$ & \num{1.687e-02} & \num{1.554e-02} & \num{1.541e-02} & \num{1.540e-02} \\
\hline
\end{tabular}
}
\end{table}

\bibliographystyle{unsrtnat}
\bibliography{refs}  

\begin{thebibliography}{44}
\providecommand{\natexlab}[1]{#1}
\providecommand{\url}[1]{\texttt{#1}}
\expandafter\ifx\csname urlstyle\endcsname\relax
  \providecommand{\doi}[1]{doi: #1}\else
  \providecommand{\doi}{doi: \begingroup \urlstyle{rm}\Url}\fi

\bibitem[Bensoussan et~al.(2011)Bensoussan, Lions, and
  Papanicolaou]{Bensoussan2011}
A.~Bensoussan, J.-L. Lions, and G.~Papanicolaou.
\newblock \emph{Asymptotic analysis for periodic structures}.
\newblock AMS Chelsea Publishing, Providence, RI, 2011.
\newblock ISBN 978-0-8218-5324-5.
\newblock \doi{10.1090/chel/374}.
\newblock Corrected reprint of the 1978 original [MR0503330].

\bibitem[Dal~Maso(1993)]{DalMaso1993}
Gianni Dal~Maso.
\newblock \emph{An introduction to {$\Gamma$}-convergence}, volume~8 of
  \emph{Progress in Nonlinear Differential Equations and their Applications}.
\newblock Birkh\"{a}user Boston, Inc., Boston, MA, 1993.
\newblock ISBN 0-8176-3679-X.
\newblock \doi{10.1007/978-1-4612-0327-8}.

\bibitem[Jikov et~al.(1994)Jikov, Kozlov, and Ole\u{\i}nik]{Jikov1994}
V.~V. Jikov, S.~M. Kozlov, and O.~A. Ole\u{\i}nik.
\newblock \emph{Homogenization of differential operators and integral
  functionals}.
\newblock Springer-Verlag, Berlin, 1994.
\newblock ISBN 3-540-54809-2.
\newblock \doi{10.1007/978-3-642-84659-5}.
\newblock Translated from the Russian by G. A. Yosifian [G. A. Iosif\cprime
  yan].

\bibitem[Pankov(1997)]{Pankov1997}
Alexander Pankov.
\newblock \emph{{$G$}-convergence and homogenization of nonlinear partial
  differential operators}, volume 422 of \emph{Mathematics and its
  Applications}.
\newblock Kluwer Academic Publishers, Dordrecht, 1997.
\newblock ISBN 0-7923-4720-X.
\newblock \doi{10.1007/978-94-015-8957-4}.

\bibitem[Conca and Vanninathan(1997)]{Conca1997}
Carlos Conca and Muthusamy Vanninathan.
\newblock Homogenization of periodic structures via {B}loch decomposition.
\newblock \emph{SIAM Journal on Applied Mathematics}, 57\penalty0 (6):\penalty0
  1639--1659, 1997.
\newblock ISSN 0036-1399.
\newblock \doi{10.1137/S0036139995294743}.

\bibitem[Cioranescu and Donato(1999)]{Cioranescu1999}
Doina Cioranescu and Patrizia Donato.
\newblock \emph{An introduction to homogenization}, volume~17 of \emph{Oxford
  Lecture Series in Mathematics and its Applications}.
\newblock The Clarendon Press, Oxford University Press, New York, 1999.
\newblock ISBN 0-19-856554-2.

\bibitem[Cioranescu et~al.(2008)Cioranescu, Damlamian, and
  Griso]{Cioranescu2008}
D.~Cioranescu, A.~Damlamian, and G.~Griso.
\newblock The periodic unfolding method in homogenization.
\newblock \emph{SIAM Journal on Mathematical Analysis}, 40\penalty0
  (4):\penalty0 1585--1620, 2008.
\newblock ISSN 0036-1410.
\newblock \doi{10.1137/080713148}.

\bibitem[Tartar(2009)]{Tartar2009}
Luc Tartar.
\newblock \emph{The general theory of homogenization}, volume~7 of
  \emph{Lecture Notes of the Unione Matematica Italiana}.
\newblock Springer-Verlag, Berlin; UMI, Bologna, 2009.
\newblock ISBN 978-3-642-05194-4.
\newblock \doi{10.1007/978-3-642-05195-1}.
\newblock A personalized introduction.

\bibitem[Shen(2018)]{Shen2018}
Zhongwei Shen.
\newblock \emph{Periodic homogenization of elliptic systems}, volume 269 of
  \emph{Operator Theory: Advances and Applications}.
\newblock Birkh\"{a}user/Springer, Cham, 2018.
\newblock ISBN 978-3-319-91213-4; 978-3-319-91214-1.
\newblock \doi{10.1007/978-3-319-91214-1}.
\newblock Advances in Partial Differential Equations (Basel).

\bibitem[Armstrong et~al.(2019)Armstrong, Kuusi, and Mourrat]{Armstrong2019}
Scott Armstrong, Tuomo Kuusi, and Jean-Christophe Mourrat.
\newblock \emph{Quantitative stochastic homogenization and large-scale
  regularity}, volume 352 of \emph{Grundlehren der Mathematischen
  Wissenschaften [Fundamental Principles of Mathematical Sciences]}.
\newblock Springer, Cham, 2019.
\newblock ISBN 978-3-030-15544-5; 978-3-030-15545-2; 978-3-030-15547-6.
\newblock \doi{10.1007/978-3-030-15545-2}.

\bibitem[Hou and Wu(1997)]{Hou1997}
Thomas~Y. Hou and Xiao-Hui Wu.
\newblock A multiscale finite element method for elliptic problems in composite
  materials and porous media.
\newblock \emph{Journal of Computational Physics}, 134\penalty0 (1):\penalty0
  169--189, 1997.
\newblock ISSN 0021-9991.
\newblock \doi{10.1006/jcph.1997.5682}.

\bibitem[Hou et~al.(1999)Hou, Wu, and Cai]{Hou1999}
Thomas~Y. Hou, Xiao-Hui Wu, and Zhiqiang Cai.
\newblock Convergence of a multiscale finite element method for elliptic
  problems with rapidly oscillating coefficients.
\newblock \emph{Mathematics of Computation}, 68\penalty0 (227):\penalty0
  913--943, 1999.
\newblock ISSN 0025-5718.
\newblock \doi{10.1090/S0025-5718-99-01077-7}.

\bibitem[Chen and Hou(2003)]{Chen2003}
Zhiming Chen and Thomas~Y. Hou.
\newblock A mixed multiscale finite element method for elliptic problems with
  oscillating coefficients.
\newblock \emph{Mathematics of Computation}, 72\penalty0 (242):\penalty0
  541--576, 2003.
\newblock ISSN 0025-5718.
\newblock \doi{10.1090/S0025-5718-02-01441-2}.

\bibitem[Efendiev and Hou(2009)]{Efendiev2009}
Yalchin Efendiev and Thomas~Y. Hou.
\newblock \emph{Multiscale finite element methods}, volume~4 of \emph{Surveys
  and Tutorials in the Applied Mathematical Sciences}.
\newblock Springer, New York, 2009.
\newblock ISBN 978-0-387-09495-3.
\newblock Theory and applications.

\bibitem[E and Engquist(2003)]{Weinan2003}
Weinan E and Bjorn Engquist.
\newblock The heterogeneous multiscale methods.
\newblock \emph{Communications in Mathematical Sciences}, 1\penalty0
  (1):\penalty0 87--132, 2003.
\newblock ISSN 1539-6746.

\bibitem[E et~al.(2005)E, Ming, and Zhang]{Weinan2005}
Weinan E, Pingbing Ming, and Pingwen Zhang.
\newblock Analysis of the heterogeneous multiscale method for elliptic
  homogenization problems.
\newblock \emph{Journal of the American Mathematical Society}, 18\penalty0
  (1):\penalty0 121--156, 2005.
\newblock ISSN 0894-0347.
\newblock \doi{10.1090/S0894-0347-04-00469-2}.

\bibitem[Abdulle et~al.(2012)Abdulle, E, Engquist, and
  Vanden-Eijnden]{Abdulle2012}
Assyr Abdulle, Weinan E, Bj\"{o}rn Engquist, and Eric Vanden-Eijnden.
\newblock The heterogeneous multiscale method.
\newblock \emph{Acta Numerica}, 21:\penalty0 1--87, 2012.
\newblock ISSN 0962-4929.
\newblock \doi{10.1017/S0962492912000025}.

\bibitem[Hughes(1995)]{Hughes1995}
Thomas J.~R. Hughes.
\newblock Multiscale phenomena: {G}reen's functions, the
  {D}irichlet-to-{N}eumann formulation, subgrid scale models, bubbles and the
  origins of stabilized methods.
\newblock \emph{Computer Methods in Applied Mechanics and Engineering},
  127\penalty0 (1-4):\penalty0 387--401, 1995.
\newblock ISSN 0045-7825.
\newblock \doi{10.1016/0045-7825(95)00844-9}.

\bibitem[Brezzi et~al.(1997)Brezzi, Franca, Hughes, and Russo]{Brezzi1997}
F.~Brezzi, L.~P. Franca, T.~J.~R. Hughes, and A.~Russo.
\newblock {$b=\int g$}.
\newblock \emph{Computer Methods in Applied Mechanics and Engineering},
  145\penalty0 (3-4):\penalty0 329--339, 1997.
\newblock ISSN 0045-7825.
\newblock \doi{10.1016/S0045-7825(96)01221-2}.

\bibitem[Hughes and Sangalli(2007)]{Hughes2007}
T.~J.~R. Hughes and G.~Sangalli.
\newblock Variational multiscale analysis: the fine-scale {G}reen's function,
  projection, optimization, localization, and stabilized methods.
\newblock \emph{SIAM Journal on Numerical Analysis}, 45\penalty0 (2):\penalty0
  539--557, 2007.
\newblock ISSN 0036-1429.
\newblock \doi{10.1137/050645646}.

\bibitem[Babuska and Lipton(2011)]{Babuska2011}
Ivo Babuska and Robert Lipton.
\newblock Optimal local approximation spaces for generalized finite element
  methods with application to multiscale problems.
\newblock \emph{Multiscale Modeling \& Simulation. A SIAM Interdisciplinary
  Journal}, 9\penalty0 (1):\penalty0 373--406, 2011.
\newblock ISSN 1540-3459.
\newblock \doi{10.1137/100791051}.

\bibitem[Babu\v{s}ka et~al.(2020)Babu\v{s}ka, Lipton, Sinz, and
  Stuebner]{Babuska2020}
Ivo Babu\v{s}ka, Robert Lipton, Paul Sinz, and Michael Stuebner.
\newblock Multiscale-spectral {GFEM} and optimal oversampling.
\newblock \emph{Computer Methods in Applied Mechanics and Engineering},
  364:\penalty0 112960, 28, 2020.
\newblock ISSN 0045-7825.
\newblock \doi{10.1016/j.cma.2020.112960}.

\bibitem[Efendiev et~al.(2013)Efendiev, Galvis, and Hou]{Efendiev2013}
Yalchin Efendiev, Juan Galvis, and Thomas~Y. Hou.
\newblock Generalized multiscale finite element methods ({GM}s{FEM}).
\newblock \emph{Journal of Computational Physics}, 251:\penalty0 116--135,
  2013.
\newblock ISSN 0021-9991.
\newblock \doi{10.1016/j.jcp.2013.04.045}.

\bibitem[Chung et~al.(2014)Chung, Efendiev, and Leung]{Chung2014}
Eric~T. Chung, Yalchin Efendiev, and Wing~Tat Leung.
\newblock Generalized multiscale finite element methods for wave propagation in
  heterogeneous media.
\newblock \emph{Multiscale Modeling \& Simulation. A SIAM Interdisciplinary
  Journal}, 12\penalty0 (4):\penalty0 1691--1721, 2014.
\newblock ISSN 1540-3459.
\newblock \doi{10.1137/130926675}.

\bibitem[Chung et~al.(2016)Chung, Efendiev, and Hou]{Chung2016}
Eric Chung, Yalchin Efendiev, and Thomas~Y. Hou.
\newblock Adaptive multiscale model reduction with generalized multiscale
  finite element methods.
\newblock \emph{Journal of Computational Physics}, 320:\penalty0 69--95, 2016.
\newblock ISSN 0021-9991.
\newblock \doi{10.1016/j.jcp.2016.04.054}.

\bibitem[M\aa~lqvist and Peterseim(2014)]{Maalqvist2014}
Axel M\aa~lqvist and Daniel Peterseim.
\newblock Localization of elliptic multiscale problems.
\newblock \emph{Mathematics of Computation}, 83\penalty0 (290):\penalty0
  2583--2603, 2014.
\newblock ISSN 0025-5718.
\newblock \doi{10.1090/S0025-5718-2014-02868-8}.

\bibitem[Henning and M\aa~lqvist(2014)]{Henning2014}
Patrick Henning and Axel M\aa~lqvist.
\newblock Localized orthogonal decomposition techniques for boundary value
  problems.
\newblock \emph{SIAM Journal on Scientific Computing}, 36\penalty0
  (4):\penalty0 A1609--A1634, 2014.
\newblock ISSN 1064-8275.
\newblock \doi{10.1137/130933198}.

\bibitem[Altmann et~al.(2021)Altmann, Henning, and Peterseim]{Altmann2021}
Robert Altmann, Patrick Henning, and Daniel Peterseim.
\newblock Numerical homogenization beyond scale separation.
\newblock \emph{Acta Numerica}, 30:\penalty0 1--86, 2021.
\newblock ISSN 0962-4929.
\newblock \doi{10.1017/S0962492921000015}.

\bibitem[Hellman and M\aa~lqvist(2017)]{Hellman2017}
Fredrik Hellman and Axel M\aa~lqvist.
\newblock Contrast independent localization of multiscale problems.
\newblock \emph{Multiscale Modeling \& Simulation. A SIAM Interdisciplinary
  Journal}, 15\penalty0 (4):\penalty0 1325--1355, 2017.
\newblock ISSN 1540-3459.
\newblock \doi{10.1137/16M1100460}.

\bibitem[M\aa~lqvist and Peterseim(2021)]{Maalqvist2021}
A.~M\aa~lqvist and D.~Peterseim.
\newblock \emph{Numerical homogenization by localized orthogonal
  decomposition}, volume~5 of \emph{SIAM Spotlights}.
\newblock Society for Industrial and Applied Mathematics (SIAM), Philadelphia,
  PA, 2021.
\newblock ISBN 978-1-611976-44-1.

\bibitem[Chung et~al.(2018{\natexlab{a}})Chung, Efendiev, and Leung]{Chung2018}
Eric~T. Chung, Yalchin Efendiev, and Wing~Tat Leung.
\newblock Constraint energy minimizing generalized multiscale finite element
  method.
\newblock \emph{Computer Methods in Applied Mechanics and Engineering},
  339:\penalty0 298--319, 2018{\natexlab{a}}.
\newblock ISSN 0045-7825.
\newblock \doi{10.1016/j.cma.2018.04.010}.

\bibitem[Vasilyeva et~al.(2019)Vasilyeva, Chung, Efendiev, and
  Kim]{Vasilyeva2019}
Maria Vasilyeva, Eric~T. Chung, Yalchin Efendiev, and Jihoon Kim.
\newblock Constrained energy minimization based upscaling for coupled flow and
  mechanics.
\newblock \emph{Journal of Computational Physics}, 376:\penalty0 660--674,
  2019.
\newblock ISSN 0021-9991.
\newblock \doi{10.1016/j.jcp.2018.09.054}.

\bibitem[Wang et~al.(2021)Wang, Chung, and Zhao]{Wang2021}
Yiran Wang, Eric Chung, and Lina Zhao.
\newblock Constraint energy minimization generalized multiscale finite element
  method in mixed formulation for parabolic equations.
\newblock \emph{Mathematics and Computers in Simulation}, 188:\penalty0
  455--475, 2021.

\bibitem[Chung and Pun(2020)]{Chung2020}
Eric Chung and Sai-Mang Pun.
\newblock Computational multiscale methods for first-order wave equation using
  mixed {CEM}-{GM}s{FEM}.
\newblock \emph{Journal of Computational Physics}, 409:\penalty0 109359, 13,
  2020.
\newblock ISSN 0021-9991.
\newblock \doi{10.1016/j.jcp.2020.109359}.

\bibitem[Bellieud and Bouchitt\'{e}(1998)]{Bellieud1998}
Michel Bellieud and Guy Bouchitt\'{e}.
\newblock Homogenization of elliptic problems in a fiber reinforced structure.
  {N}onlocal effects.
\newblock \emph{Annali della Scuola Normale Superiore di Pisa. Classe di
  Scienze. Serie IV}, 26\penalty0 (3):\penalty0 407--436, 1998.
\newblock ISSN 0391-173X.

\bibitem[Briane(2002)]{Briane2002}
Marc Briane.
\newblock Homogenization of non-uniformly bounded operators: critical barrier
  for nonlocal effects.
\newblock \emph{Archive for Rational Mechanics and Analysis}, 164\penalty0
  (1):\penalty0 73--101, 2002.
\newblock ISSN 0003-9527.
\newblock \doi{10.1007/s002050200196}.

\bibitem[Du et~al.(2020)Du, Engquist, and Tian]{Du2020}
Qiang Du, Bjorn Engquist, and Xiaochuan Tian.
\newblock Multiscale modeling, homogenization and nonlocal effects:
  mathematical and computational issues.
\newblock In \emph{75 years of mathematics of computation}, volume 754 of
  \emph{Contemp. Math.}, pages 115--139. Amer. Math. Soc., [Providence], RI,
  2020.
\newblock \doi{10.1090/conm/754/15175}.

\bibitem[Scott and Zhang(1990)]{Scott1990}
L.~Ridgway Scott and Shangyou Zhang.
\newblock Finite element interpolation of nonsmooth functions satisfying
  boundary conditions.
\newblock \emph{Mathematics of Computation}, 54\penalty0 (190):\penalty0
  483--493, 1990.
\newblock ISSN 0025-5718.
\newblock \doi{10.2307/2008497}.

\bibitem[Brenner and Scott(2008)]{Brenner2008}
Susanne~C. Brenner and L.~Ridgway Scott.
\newblock \emph{The mathematical theory of finite element methods}, volume~15
  of \emph{Texts in Applied Mathematics}.
\newblock Springer, New York, third edition, 2008.
\newblock ISBN 978-0-387-75933-3.
\newblock \doi{10.1007/978-0-387-75934-0}.

\bibitem[Brezzi and Fortin(1991)]{Brezzi1991}
Franco Brezzi and Michel Fortin.
\newblock \emph{Mixed and hybrid finite element methods}, volume~15 of
  \emph{Springer Series in Computational Mathematics}.
\newblock Springer-Verlag, New York, 1991.
\newblock ISBN 0-387-97582-9.
\newblock \doi{10.1007/978-1-4612-3172-1}.

\bibitem[Harris et~al.(2020)Harris, Millman, van~der Walt, Gommers, Virtanen,
  Cournapeau, Wieser, Taylor, Berg, Smith, Kern, Picus, Hoyer, van Kerkwijk,
  Brett, Haldane, del R{\'{i}}o, Wiebe, Peterson, G{\'{e}}rard-Marchant,
  Sheppard, Reddy, Weckesser, Abbasi, Gohlke, and Oliphant]{Harris2020}
Charles~R. Harris, K.~Jarrod Millman, St{\'{e}}fan~J. van~der Walt, Ralf
  Gommers, Pauli Virtanen, David Cournapeau, Eric Wieser, Julian Taylor,
  Sebastian Berg, Nathaniel~J. Smith, Robert Kern, Matti Picus, Stephan Hoyer,
  Marten~H. van Kerkwijk, Matthew Brett, Allan Haldane, Jaime~Fern{\'{a}}ndez
  del R{\'{i}}o, Mark Wiebe, Pearu Peterson, Pierre G{\'{e}}rard-Marchant,
  Kevin Sheppard, Tyler Reddy, Warren Weckesser, Hameer Abbasi, Christoph
  Gohlke, and Travis~E. Oliphant.
\newblock Array programming with {NumPy}.
\newblock \emph{Nature}, 585\penalty0 (7825):\penalty0 357--362, September
  2020.
\newblock \doi{10.1038/s41586-020-2649-2}.

\bibitem[Virtanen et~al.(2020)Virtanen, Gommers, Oliphant, Haberland, Reddy,
  Cournapeau, Burovski, Peterson, Weckesser, Bright, {van der Walt}, Brett,
  Wilson, Millman, Mayorov, Nelson, Jones, Kern, Larson, Carey, Polat, Feng,
  Moore, {VanderPlas}, Laxalde, Perktold, Cimrman, Henriksen, Quintero, Harris,
  Archibald, Ribeiro, Pedregosa, {van Mulbregt}, and {SciPy 1.0
  Contributors}]{Virtanen2020}
Pauli Virtanen, Ralf Gommers, Travis~E. Oliphant, Matt Haberland, Tyler Reddy,
  David Cournapeau, Evgeni Burovski, Pearu Peterson, Warren Weckesser, Jonathan
  Bright, St{\'e}fan~J. {van der Walt}, Matthew Brett, Joshua Wilson, K.~Jarrod
  Millman, Nikolay Mayorov, Andrew R.~J. Nelson, Eric Jones, Robert Kern, Eric
  Larson, C~J Carey, {\.I}lhan Polat, Yu~Feng, Eric~W. Moore, Jake
  {VanderPlas}, Denis Laxalde, Josef Perktold, Robert Cimrman, Ian Henriksen,
  E.~A. Quintero, Charles~R. Harris, Anne~M. Archibald, Ant{\^o}nio~H. Ribeiro,
  Fabian Pedregosa, Paul {van Mulbregt}, and {SciPy 1.0 Contributors}.
\newblock {{SciPy} 1.0: Fundamental Algorithms for Scientific Computing in
  Python}.
\newblock \emph{Nature Methods}, 17:\penalty0 261--272, 2020.
\newblock \doi{10.1038/s41592-019-0686-2}.

\bibitem[Chung et~al.(2018{\natexlab{b}})Chung, Efendiev, Leung, Vasilyeva, and
  Wang]{Chung2018a}
Eric~T. Chung, Yalchin Efendiev, Wing~Tat Leung, Maria Vasilyeva, and Yating
  Wang.
\newblock Non-local multi-continua upscaling for flows in heterogeneous
  fractured media.
\newblock \emph{Journal of Computational Physics}, 372:\penalty0 22--34,
  2018{\natexlab{b}}.
\newblock ISSN 0021-9991.
\newblock \doi{10.1016/j.jcp.2018.05.038}.

\bibitem[Ivrii(2016)]{Ivrii2016}
Victor Ivrii.
\newblock 100 years of {W}eyl's law.
\newblock \emph{Bulletin of Mathematical Sciences}, 6\penalty0 (3):\penalty0
  379--452, 2016.
\newblock ISSN 1664-3607.
\newblock \doi{10.1007/s13373-016-0089-y}.

\end{thebibliography}






\end{document}